\documentclass[12pt]{iopart}
\usepackage{iopams}
\usepackage{setstack}
\usepackage{amsfonts}
\usepackage{amsthm}

\usepackage{amssymb}
\usepackage{graphicx}
\usepackage{epsfig}
\usepackage{psfrag}
\usepackage{hyperref}
\usepackage{fancyvrb}
\usepackage{color}

\newtheorem{theorem}{Theorem}

\newtheorem{corollary}[theorem]{Corollary}

\newtheorem{definition}[theorem]{Definition}

\newtheorem{lemma}[theorem]{Lemma}
\newtheorem{notation}[theorem]{Notation}

\newtheorem{remark}[theorem]{Remark}

\def\be{\begin{equation}}
\def\ee{\end{equation}}
\def\bea{\begin{eqnarray}}
\def\eea{\end{eqnarray}}

\def\hsm1{\hspace{-1mm}}

\newcommand{\correction}[2]{#2}
\newcommand{\ch}[1]{\includegraphics[height=0.37cm]{#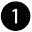}}

\usepackage[msc-links]{amsrefs}
\makeatletter
\providecommand\underarrow@[3]{%
  \vtop{\ialign{##\crcr$\m@th\hfil#2#3\hfil$\crcr
  \noalign{\nointerlineskip\kern.12\baselineskip}#1#2\crcr}}}
\providecommand{\underrightarrow}{%
  \mathpalette{\underarrow@\rightarrowfill@}}
\providecommand\rightarrowfill@{\arrowfill@\relbar\relbar\rightarrow}
\providecommand\arrowfill@[4]{%
  $\m@th\thickmuskip0mu\medmuskip\thickmuskip\thinmuskip\thickmuskip
   \relax#4#1\mkern-7mu%
   \cleaders\hbox{$#4\mkern-2mu#2\mkern-2mu$}\hfill
   \mkern-7mu#3$%
}
\makeatother

\begin{document}
\title[Persistence of NHIMs in the absence of rate conditions]{Persistence of normally hyperbolic invariant manifolds in the absence of rate conditions}

\author{Maciej J. Capi\'nski\footnote{Partially supported by the NCN grants  2015/19/B/ST1/01454, 2016/22/A/ST1/00077, 2016/21/B/ST1/02453 and by the Faculty of Applied Mathematics AGH UST statutory tasks 11.11.420.004 within subsidy of Ministry of Science and Higher Education.}}
\ead{maciej.capinski@agh.edu.pl}
\address{AGH University of Science and Technology, al. Mickiewicza 30, 30-059 Krak\'ow, Poland}

\author{Hieronim Kubica\footnote{Partially supported by the NCN grant  2016/21/B/ST1/02453, and by the Faculty of Applied Mathematics AGH UST dean grant for PhD students and young researchers within subsidy of Ministry of Science and Higher Education.} }
\ead{kubica@agh.edu.pl}
\address{AGH University of Science and Technology, al. Mickiewicza 30, 30-059 Krak\'ow, Poland}

\begin{abstract}
We consider perturbations of normally hyperbolic invariant manifolds, under which they can lose their hyperbolic properties. 
We show that if the perturbed map which drives the dynamical system preserves the properties of topological expansion and contraction,
then the manifold is perturbed to an invariant set. The main feature is that our results do not require the rate conditions to hold after the perturbation. In this case the manifold can be perturbed to an invariant set, which is not a topological manifold. We work in the setting of nonorientable Banach vector bundles, without needing to assume invertibility of the map.

\end{abstract}





\section{Introduction}

We will be investigating the persistence under perturbations of invariant sets that are associated with normally hyperbolic invariant manifolds (NHIMs). These perturbations will be such that the manifolds lose their hyperbolic properties. 

To be more precise, a manifold $\Lambda$ is said to be a NHIM if it is invariant for a dynamical system and there is a splitting of the state space into three invariant subbundles. One is the tangent bundle to $\Lambda$, the second is the unstable bundle and the third is the stable bundle. The dynamics on the stable bundle is contracting and on the unstable bundle -- expanding. The key feature for $\Lambda$ to be normally hyperbolic is that the dynamics on the bundle tangent to $\Lambda$ is weaker than the dynamics on the stable and the unstable bundles. The property of the dominance of the dynamics on the stable/unstable bundles over the tangent bundle is formulated in terms the rate conditions, introduced by Fenichel \cite{Fenichel1,Fenichel4,Fenichel3,Fenichel2}, Hirsch, Pugh, Shub \cite{HPS}, and later developed by Chaperon \cite{Chap1,Chap2,Chap3}.

The main property of NHIMs is that they persist under perturbations. As long as the rate conditions hold, the manifold is present.
There are examples though \cite{GonDroJun2014, GonJun2015, Jarnik, TerTodKom2015} for which, in the absence of rate conditions, an invariant manifold can be destroyed to a set which is not even a topological manifold. However, this does not mean that the manifold vanishes or that it is completely destroyed. 

This problem has been studied by Floer in \cite{Floer2,Floer}. He introduced a method, which allowed him to establish continuation of NHIMs to invariant sets which preserve the cohomology ring of the manifold under perturbation. 
We take a different approach, which is based on a good topological alignment expressed by homotopy conditions. We establish existence of an invariant set whose projection onto the base manifold $\Lambda$ is equal to the whole $\Lambda$. The advantage of our method is that it does not rely on the prior existence of a normally hyperbolic invariant manifold and neither does it use perturbation theory. Moreover, we prove a continuation theorem for invariant sets of continuous one-parameter families of maps under the assumption of correct topological alignment. To be more precise, we show that if we extend the system to include the parameter, then, in such an extended phase space, there exists a compact connected component consisting of points belonging to the invariant sets of maps corresponding to varying parameter values.

Our result does not contradict the work of Ma\~n\'e \cite{Mane}. He shows that if a manifold is persistent, then it has to be normally hyperbolic. What we establish though is not persistence of manifolds, but persistence of sets. 
In fact, we do not need the normally hyperbolic invariant manifold to exist. If we have a family of maps that satisfy our topological assumptions, then we will have persistence of the family of their invariant sets.

The main features of our results are the following. Our work is written in the context of Banach vector bundles, without any orientability assumptions. We establish the existence of non-empty invariant sets for discrete dynamical systems. These sets are not only non-empty, but also have non-empty intersections with each fiber of the vector bundle, meaning that they project surjectively onto the base manifold. We do not need to assume that our map is invertible. We do not need a normally hyperbolic manifold prior to perturbation; our method can be used to establish the existence of invariant sets with `topologically normally hyperbolic' properties. If the assumptions of our theorems are verified, then we obtain the existence of  invariant sets within their specific, explicitly given neighborhood. Verification can be performed using rigorous, interval arithmetic numerics, leading to computer assisted proofs. Our results are written in the context of discrete dynamical systems, but they can also be applied to ODEs by considering a time-shift map.

Our approach is based on the method of covering relations \cite{GZ2,Z-conecond,GZ1}. The following results can be thought of as its generalization to vector bundles. Covering relations have proven to be a useful tool that, combined with cone conditions, leads to geometric proofs of normally hyperbolic invariant manifold theorems \cite{CZ1,CZ2}. These results, however, rely also on a form of rate conditions, expressed in terms of cone conditions. Another result in this flavour is \cite{BB}, which contains another geometric version of the normally hyperbolic invariant manifold theorem. Although again, it relies on rate conditions and on perturbative methods. Our work is closely related to \cite{Cap}, which can also be applied in the absence of rate conditions. The difference is that in \cite{Cap} only the case of trivial vector bundles and invertible maps was considered. This paper can be thought of as a generalization of \cite{Cap} to the setting of general, possibly nonoriantable vector bundles, without the assumption on invertibility of the map. Moreover, in the present work we obtain a continuation result, which states that in the state space extended to include a parameter, the invariant sets for a family of maps 
contain a connected component which links them together.


The paper is organized as follows. Section \ref{sec:preliminaries} contains preliminaries. There we set up our notations used for vector bundles and introduce the notion of an intersection number. The intersection number is a standard tool in differential topology, which can be used to detect intersections of manifolds based on their homotopy properties. 
In section \ref{sec:main} we state our main results, which are formulated in Theorems \ref{th:forward-invariance}, \ref{th:invariance}, \ref{th:forward-continuation} and \ref{th:continuation}, we also show that normal hyperbolicity implies the assumptions of the theorems, and give an example of application. Sections \ref{proof:forward-invariance}, \ref{proof:invariance}, \ref{proof:forward-continuation} and \ref{proof:continuation} contain the proofs of the four theorems. Section \ref{sec:nhim-covers} contains the proof of the fact that normal hyperbolicity implies topological covering. Section \ref{sec:ack} contains acknowledgements. To keep the paper self-contained and also since our approach to the intersection number is slightly non-standard (we allow our manifolds to have boundaries), we add the construction of the intersection number in  \ref{sec:app}.
\section{Preliminaries\label{sec:preliminaries}}

\subsection{Notations}

For a set $A$ \correction{comment 11}{in some topological space }we use $\partial A$ to denote its boundary, $\overline{A}$ to denote its closure, and $\mathrm{int}A$ to denote its interior. We write $\#A$ to denote the cardinality of $A$.

For a compact connected manifold $\Lambda$ and a continuous map $f:\Lambda
\to\Lambda$ we shall use $\deg_{2}f$ to denote the degree modulo $2$ of $f$
(see \cite{Hirsch} for details).

For two sets $A, B\subset\mathbb{R}^{n}$ we shall use $\mathrm{dist}(A,B)$ to denote the distance between them. We will use the
notation $B_{\mathbb{R}^{n}}(x,r)$ to stand for an open ball centered at $x$ of
radius $r$ in $\mathbb{R}^{n}$.

\subsection{Banach vector bundles}

In this section we set up some notations for Banach vector bundles, which will
be used throughout the paper.

Let $\Lambda$ be a topological space. We recall that a \emph{vector bundle of
rank $k$} over $\Lambda$ is a topological space $E$ together with a surjective
continuous map $\pi:E\rightarrow\Lambda$ satisfying the following conditions:

\begin{enumerate}
\item For all $\theta\in\Lambda$, the \emph{fiber} $E_{\theta}:=\pi
^{-1}(\theta)$ over $\theta$ is a $k$-dimensional vector space.

\item For every $\theta\in\Lambda$ there exists an open neighborhood
$U_{\theta}$ of $\theta$ in $\Lambda$ and a homeomorphism 
\[
\Phi_{\theta}: \pi^{-1}(U_{\theta}) \to U_{\theta}\times\mathbb{R}^{k},
\]
called a \emph{local trivialization of $E$ over $U_{\theta}$}, such that:

\begin{itemize}
\item $\pi_{\theta}\circ\Phi_{\theta}=\pi$, where $\pi_{\theta}:U_{\theta
}\times\mathbb{R}^{k}\rightarrow U_{\theta}$ is the projection on $U_{\theta}$.

\item For every $\lambda\in U_{\theta}$ the restriction of $\Phi_{\theta}$ to
the fiber $E_{\lambda}$
\[
\Phi_{\theta}|_{E_{\lambda}}: E_{\lambda}\to\{ \lambda\} \times\mathbb{R}^{k}
\cong\mathbb{R}^{k}%
\]
is a vector space isomorphism. The set $U_{\theta}$ is called the \emph{base}
of the local trivialization $\Phi_{\theta}$.
\end{itemize}
\end{enumerate}

The space $E$ is called the \emph{total space} of the bundle, $\Lambda$ is
called its \emph{base}, and $\pi$ is its \emph{projection}. In our paper we
will be dealing with smooth vector bundles, meaning that $\Lambda$ and $E$
will be smooth manifolds and the projection will be a smooth map.

When $\Phi_{\theta_{1}}:\pi^{-1}\left(  U_{\theta_{1}}\right)  \rightarrow
U_{\theta_{1}}\times\mathbb{R}^{k}$ and $\Phi_{\theta_{2}}:\pi^{-1}\left(
U_{\theta_{2}}\right)  \rightarrow U_{\theta_{2}}\times\mathbb{R}^{k}$ are two
local trivializations of $E$ such that $U_{\theta_{1}}\cap U_{\theta_{2}}%
\neq\emptyset$, and $\lambda\in U_{\theta_{1}}\cap U_{\theta_{2}}$, the
function
\[
\left(  \pi_{\mathbb{R}^{k}}\circ\Phi_{\theta_{2}}|_{{E_{\lambda}}}\right)
\circ\left(  \pi_{\mathbb{R}^{k}}\circ\Phi_{\theta_{1}}|_{E_{\lambda}}\right)
^{-1}:\mathbb{R}^{k}\rightarrow\mathbb{R}^{k}%
\]
is called a \emph{transition function} between local trivializations.

If we are given a vector bundle $\pi:E\rightarrow\Lambda$ with a fixed
collection of local trivializations $\{\Phi_{\theta}:\pi^{-1}(U_{\theta
})\rightarrow U_{\theta}\times\mathbb{R}^{k}\}$ whose bases form an open cover
$\mathcal{U}=\{U_{\theta}\}$ of $\Lambda$, then we call it a \emph{Banach
vector bundle} provided that all transition functions between local
trivializations with overlapping bases are isometries.

Henceforth, we shall assume that every vector bundle we work with is a
Banach vector bundle even if it is not explicitly pronounced.

For Banach vector bundles we are able to introduce a meaningful notion of a
norm on fibers as follows. For every $v\in E$ such that $\pi(v)\in U_{\theta}%
$, where $U_{\theta}$ is trivialized by $\Phi_{\theta}$, we define
\[
\Vert v\Vert:=\Vert{\pi_{\mathbb{R}^{k}}\circ\Phi_{\theta}(v)}\Vert
_{\mathrm{Eucl}},
\]
where $\Vert\cdot\Vert_{\mathrm{Eucl}}$ is the Euclidean norm on
$\mathbb{R}^{k}$. Since all transition functions between local trivializations
with overlapping bases are isometries, we see that $\left\Vert v\right\Vert $
does not depend on the choice of $\Phi_{\theta}$.
\correction{comment 12}{\begin{remark}\label{rem:Banach}
We use the name \emph{Banach vector bundle} since in our setting the fibres are finite dimensional Banach spaces. By writing \emph{Banach vector bundle} we implicitly assume that the transition functions are isometries, which is somewhat non-standard and needs to be emphasised. Moreover, we do not consider vector bundles with infinite-dimensional fibers, which the prefix `Banach' is often assumed to imply.
\end{remark}}
\begin{remark}
\label{rem:norm-interpretation}For  $v\in E$ the notation $\Vert v\Vert$
should be understood as the norm on the fiber $E_{\pi(v)}$. (It makes no sense
to talk of a norm on $E$, since it is not a vector space.)
\end{remark}

\subsection{Whitney sum of Banach vector bundles}

Consider a smooth manifold $\Lambda$, a rank-$u$ smooth Banach vector bundle
$\pi^{u}:E^{u}\rightarrow\Lambda$ with a fixed collection of local
trivializations
\[
\left\{  \Phi_{\theta}^{u}:(\pi^{u})^{-1}(U_{\theta})\rightarrow U_{\theta
}\times\mathbb{R}^{u}|\;U_{\theta}\text{ cover }\Lambda\right\}
\]
inducing a Banach space structure on the fibers of the total space $E^{u}$ and
a rank-$s$ smooth Banach vector bundle $\pi^{s}:E^{s}\rightarrow\Lambda$ with
fixed
\[
\left\{  \Phi_{\theta}^{s}:(\pi^{s})^{-1}(V_{\theta})\rightarrow V_{\theta
}\times\mathbb{R}^{s}|\;V_{\theta}\text{ cover }\Lambda\right\}
\]
inducing a Banach space structure on the fibers of $E^{s}$.

We combine the two vector bundles in what is called a \emph{Whitney sum} to
produce a new vector bundle $E=E^{u}\oplus E^{s}$ of rank $u+s$ over $\Lambda
$, defined as
\[
E=E^{u}\oplus E^{s}:=\bigsqcup_{\theta\in\Lambda}E_{\theta}^{u}\oplus
E_{\theta}^{s},
\]
where $\bigsqcup$ stands for the disjoint union. The fiber $E_{\theta}$ of $E$
over each $\theta\in\Lambda$ is the direct sum $E_{\theta}^{u}\oplus
E_{\theta}^{s}.$ The projection $\pi:E=E^{u}\oplus E^{s}\rightarrow\Lambda$ is
the natural one.

\begin{notation}
\label{not:whitney-triple}To represent a point $v\in E=E^{u}\oplus E^{s}$ we
shall identify it with a triple $(\theta;x,y)$, where $\theta=\pi(v)$ and
$v=\left(  x,y\right)  \in E_{\theta}^{u}\oplus E_{\theta}^{s}$. In other
words, by writing
\[
v=(\theta;x,y)\in E
\]
we intend to emphasize that $x\in E_{\theta}^{u}$ and $y\in
E_{\theta}^{s}$.
\end{notation}

For $W_{\theta}\subset\Lambda$ small enough so that $\Phi_{\theta}^{u}$ and
$\Phi_{\theta}^{s}$ are both defined over $W_{\theta}$, we define the local
trivializations $\Phi_{\theta}:\pi^{-1}(W_{\theta})\rightarrow W_{\theta
}\times\mathbb{R}^{u+s}$ in the natural way. For any $\lambda\in W_{\theta}$
and$\ v=\left(  \lambda;x,y\right)  \in E$,
\[
\Phi_{\theta}(\lambda;x,y):=(\lambda;\pi_{\mathbb{R}^{u}}\circ\Phi_{\theta
}^{u}(x),\pi_{\mathbb{R}^{s}}\circ\Phi_{\theta}^{s}(y)).
\]

We will write $\left\Vert x\right\Vert _{u}$ for the norm on the fiber
$E_{\pi^{u}(x)}^{u}$ and, similarly $\left\Vert y\right\Vert _{s}$ for the
norm on the fiber $E_{\pi^{s}(y)}^{s}$.

\subsection{Intersection number\label{sec:intersection-number}}

In this section we introduce the intersection number, which will be the main
tool in our proofs. It is a standard notion in differential topology. (We
suggest \cite{GP74, Hirsch} as references.)

Let $\overline{X}$ be a compact set (subset of some \correction{comment 13}{smooth manifold}). Assume that $\mathrm{int}\overline{X}=X$
is a smooth manifold. (We do not need to assume that \correction{comment 14}{$\overline{X}$ }is a
manifold with boundary; i.e. we do not need $\partial \overline{X}$ to be a smooth manifold.) Let $Y$
be a boundaryless smooth manifold. Let $Z$ be an embedded boundaryless smooth submanifold of
$Y$, let $\overline{Z}$ be its closure in $Y$, and let $\partial Z=\overline
{Z}\setminus Z$ (which, in general, can be empty). Assume $X$ and $Z$ to be of complementary dimension with
respect to $Y$, i.e., $\dim X+\dim Z=\dim Y.$

We say that a smooth map $f:X\rightarrow Y$ \emph{is transversal to }$Z$ if%
\[
Df_{x}\left(  T_{x}X\right)  +T_{f(x)}Z=T_{f(x)}Y
\]
for all $x\in f^{-1}(Z)$. ($T_{x}X$ stands for the tangent space to $X$ at
$x$; $Df_{x}$ denotes the differential of $f$ at $x$.)

\begin{figure}[ptb]
\begin{center}
\includegraphics[width=11cm]{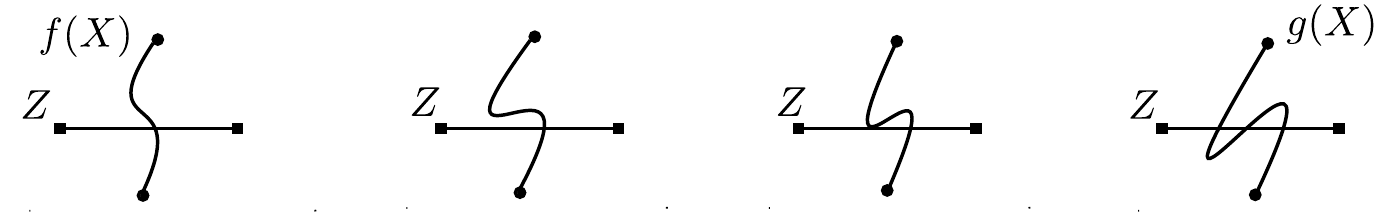}
\end{center}
\caption{Two admissible maps $f$ and $g$ homotopic through an admissible
homotopy. }%
\label{fig:inter-homotopy}%
\end{figure}

\begin{definition}
\label{def:admissible}We shall say that $f:\overline{X}\rightarrow Y$ is
\emph{admissible} if it is continuous and
\begin{eqnarray*}
f(\partial X)\cap\overline{Z}  & =\emptyset,\\
f\left(  \overline{X}\right)  \cap\partial Z  & =\emptyset.
\end{eqnarray*}

\end{definition}

\begin{definition}
\label{def:admissible-H}We shall say that a \correction{comment 15}{continuous }homotopy $H:[0,1]\times
\overline{X}\rightarrow Y$ is \emph{admissible} if (see Figure
\ref{fig:inter-homotopy})
\begin{eqnarray*}
H([0,1]\times\partial X)\cap\overline{Z}  & =\emptyset,\\
H([0,1]\times\overline{X})\cap\partial Z  & =\emptyset.
\end{eqnarray*}

\end{definition}

The\emph{ modulo }$2$\emph{ intersection number} for an admissible map
$f:\overline{X}\rightarrow Y$ and $\overline{Z}\subset Y$ is defined as a number
\[
I_{2}\left(  f,X,Z\right)  \in\left\{  0,1\right\}  ,
\]
which possesses the following properties:

\begin{itemize}
\item (Intersection number for transversal maps) If $f|_{X}$ is smooth and
transversal to $Z$ then
\[
I_{2}\left(  f,X,Z\right)  =\#f^{-1}\left(  Z\right)  \,\mathrm{mod}%
\,2.
\]

\item (Homotopy property) If $f,g$ are homotopic through an admissible
homotopy, then%
\[
I_{2}\left(  f,X,Z\right)  =I_{2}\left(  g,X,Z\right)  .
\]

\item (Intersection property) If $I_{2}\left(  f,X,Z\right)  =1$ then
$f\left(  X\right)  \cap Z$ is nonempty.

\item (Excision property) If $V$ is an open subset of $X$ such that $f\left(
X\right)  \cap Z=f\left(  V\right)  \cap Z$ and $f\left(
\partial V\right)  \cap\overline{Z}=\emptyset$ then
\[
I_{2}(f,X,Z)=I_{2}\left(  f|_{\overline{V}},V,Z\right)  .
\]

\end{itemize}

When $f:\overline{X}\rightarrow Y$, $\overline{V}\subset\overline{X}$ and
$f|_{\overline{V}}:\overline{V}\rightarrow Y$ is admissible, then we will
write $I_{2}\left(  f,V,Z\right)  $ instead of $I_{2}\left(  f|_{\overline{V}%
},V,Z\right)  $ to simplify notation.

In Figure \ref{fig:inter-homotopy} we find the intuition behind the
definition. There, while passing through an admissible homotopy, we encounter
a tangential intersection, but the number of transversal intersections is
either $1$ or $3$, so the mod $2$ intersection number is $1$. On the picture the
$f(\partial X)$ and $g(\partial X)$ are indicated by dots. These need to be
disjoint from $\overline{Z}$ throughout the admissible homotopy. The $\partial
Z$ is depicted with squares. It needs to be disjoint from the image of
$\overline{X}$ throughout the homotopy.

In the standard approach $X$ is assumed to be a compact
boundaryless manifold and $Z$ is assumed to be a closed boundaryless
submanifold of $Y$. Here we allow for $X$ and $Z$ to have boundaries, since
this will be convenient in our application. We deal with the boundary by
restricting to admissible maps and admissible homotopies, which rule out the
intersection for points from the boundaries. In such a case, the existence and
properties of the intersection number follow in the same way as the
construction for manifolds without boundary \cite{GP74, Hirsch}.

To keep the paper self-contained, and since allowing $X$ and $Z$ to have a
boundary is slightly nonstandard, we have added the construction of the
intersection number in \ref{sec:app}.

\correction{comment 17}{
\begin{remark} \label{rem:oriented-inter}
In the same way as above we can also allow $X$ and $Z$ to have boundaries in the case of the oriented intersection number. (See \cite{GP74, Hirsch} for the definition of the oriented intersection number.)
\end{remark}}

\section{Main results\label{sec:main}}

Assume that $\Lambda$ is a compact smooth $c$-dimensional manifold without
boundary, $E^{u}$, $E^{s}$ are smooth Banach vector bundles over $\Lambda$,
and that $E=E^{u}\oplus E^{s}$. We define the following sets (below and
through the reminder of the paper we use the convention from Notation
\ref{not:whitney-triple})%
\begin{eqnarray}
D  &:=\left\{  (\theta;x,y)\in E|\;\theta\in\Lambda,\ \left\Vert
x\right\Vert _{u}\leq1,\ \left\Vert y\right\Vert _{s}\leq1\right\}
,\nonumber\\
D^{-}  &:=\left\{  (\theta;x,y)\in E|\;\theta\in\Lambda,\ \left\Vert
x\right\Vert _{u}=1,\ \left\Vert y\right\Vert _{s}\leq1\right\}
,\label{eq:Ddef}\\
D^{+}  &:=\left\{  (\theta;x,y)\in E|\;\theta\in\Lambda,\ \left\Vert
x\right\Vert _{u}\leq1,\ \left\Vert y\right\Vert _{s}=1\right\}  .\nonumber
\end{eqnarray}
For $\theta\in\Lambda$ and $U\subset\Lambda$ we define the following subsets
of $E$:
\[
D_{\theta}:=D\cap E_{\theta},\qquad D_{\theta}^{-}:=D^{-}\cap E_{\theta
},\qquad D_{U}:=\bigcup_{\theta\in U}D_{\theta}.
\]
We will also use the following notation for a closed unit ball in a fiber
$E_{\theta}^{u}$
\[
B_{\theta}^{u}:=\left\{  x\in E_{\theta}^{u}|\;\left\Vert x\right\Vert
_{u}\leq1\right\}  .
\]

\subsection{Existence and continuation of invariant sets}

In this section we formulate our four main theorems. We first introduce a definition that is required to express the assumptions of our first main result. This is a generalization of the notion of `covering relations' which was introduced in \cite{GZ2,Z-conecond,GZ1}. There the covering involves a topological expansion of a set in the direction of hyperbolic expansion, and topological contraction of the set in the direction of hyperbolic contraction. Our approach is an extension of the notion to vector bundles that also have central directions associated with the base manifold. 

\begin{definition}
\label{def:covering-theta} Consider a continuous map $f:D\rightarrow E$ (not
necessarily invertible). For $\theta\in\Lambda$ we say that \emph{$D_{\theta}$
$f$-covers $D$}, denoted $D_{\theta}\overset{f}{\Longrightarrow}D$, if the following
conditions are satisfied:\correction{comment 4}{}

\begin{enumerate}
\item There exists a homotopy $h_{\theta}:[0,1]\times D_{\theta}\rightarrow E$
such that the following hold true
\begin{eqnarray*}
h_{\theta}(0,\cdot)  &  =f(\cdot),\\
h_{\theta}([0,1]\times D_{\theta}^{-})\cap D  &  =\emptyset,\\
h_{\theta}([0,1]\times D_{\theta})\cap D^{+}  &  =\emptyset.
\end{eqnarray*}

\item One of the following is satisfied:
\begin{enumerate}
\item[a.] If $u>0$, then there exists a $\Theta\in\Lambda$ (which can depend
on $\theta$) and a linear map $A_{\theta}:E_{\theta}^{u}\rightarrow E_{\Theta
}^{u}$ such that $A_{\theta}(\partial B_{\theta}^{u})\subset E_{\Theta}%
^{u}\setminus B_{\Theta}^{u}$ ($A_{\theta}$ is \emph{expanding}) and
\[
h_{\theta}(1,\left(  \theta;x,y\right)  )=(\Theta;A_{\theta}x,0)\in E_{\Theta
}.
\]

\item[b.] If $u=0$, then there exists a point $\Theta\in\Lambda$ (which can
depend on $\theta$), such that
\[
h_{\theta}(1,\left(  \theta;y\right)  )=\left(  \Theta;0\right)  \in
E_{\Theta}=E_{\Theta}^{s}.
\]
(In the above line we have omitted $x$ from the notation $\left(
\theta;x,y\right)  $ since $E^{u}$ is of dimension zero.)
\end{enumerate}
\end{enumerate}
\end{definition}

\begin{figure}[ptb]
\begin{center}
\includegraphics[width=12cm]{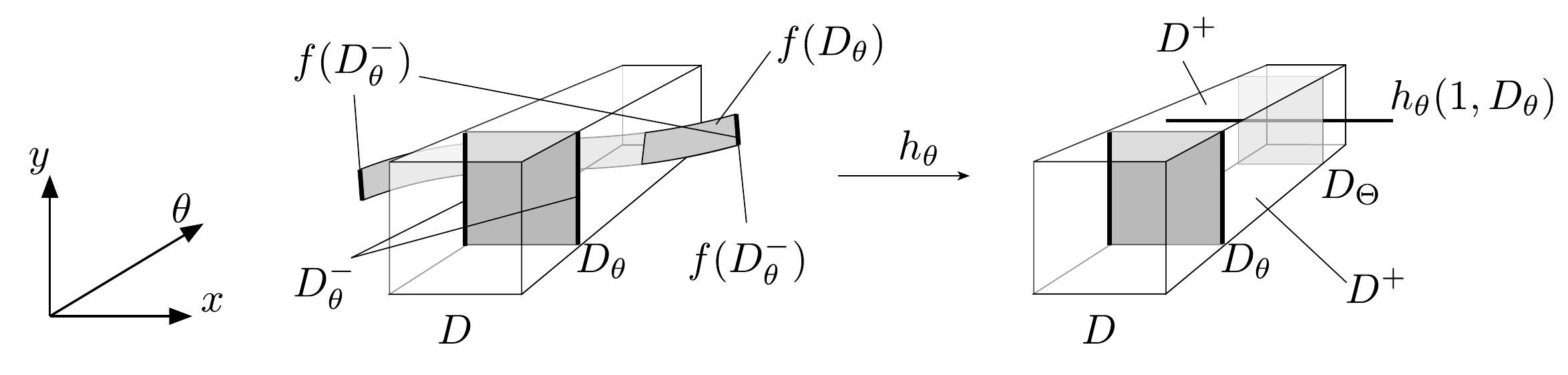}
\end{center}
\caption{Example of a homotopy from Definition \ref{def:covering-theta} of the
covering $D_{\theta}\overset{f}{\Longrightarrow}D$.}%
\label{fig:covering-theta}%
\end{figure}

The intuition behind Definition \ref{def:covering-theta} is depicted in Figure
\ref{fig:covering-theta}. There we consider $\Lambda$ to be a circle, and
$E^{u}$ and $E^{s}$ to be trivial bundles over $\Lambda$ with real one
dimensional fibers; in short, we consider $E=\mathbb{S}^{1}\times
\mathbb{R}\times\mathbb{R}$. On the plot, the front and the back sides (i.e.
$D_{\theta=0}$ and $D_{\theta=2\pi}$) of the set $D$ are identified to be the
same. For the conditions of Definition \ref{def:covering-theta} to hold we
need to have topological expansion on the $x$ coordinates. This means that the
`exit set' $D_{\theta}^{-}$ will be mapped outside of $D$. In addition, we
also need topological contraction on the coordinate $y$. This ensures that
$f(D_{\theta})$ will not intersect with $D^{+}$. We impose quite mild
conditions on the dynamics on $\theta$. It is enough that the correct
topological alignment can be pulled by a homotopy to a fiber $E_{\Theta}$. 
Note that in Definition \ref{def:covering-theta} we do not require the map to carry fibers into fibers, as is the case in the setting of normal hyperbolicity. Such assumption is not needed for any of our results in this paper.

We now formulate our first main result.

\begin{theorem}
\label{th:forward-invariance}If $f:D\rightarrow E$ is a continuous mapping and
$D_{\theta}\overset{f}{\Longrightarrow}D$ holds for every $\theta\in\Lambda$, then
for any $\theta\in\Lambda$ there exists a trajectory starting from $D_{\theta
}$, which remains in $D$ for all forward iterates, i.e., there exists $v\in
D_{\theta}$ such that $f^{m}\left(  v\right)  \in D$ for all $m\in\mathbb{N}$.
\end{theorem}

The proof is given in section \ref{proof:forward-invariance}.

Theorem \ref{th:forward-invariance} establishes the existence of points that
remain in $D$ for all iterates of a map when going forwards in time. Now we
turn to what happens also backwards in time. For this we make an additional
assumption that $\Lambda$ is a connected manifold.

\begin{definition}
\label{def:covering} Consider a continuous map $f:D\rightarrow E$ (not
necessarily invertible). We say that $D$ f\emph{-covers} $D$, denoted
$D\overset{f}{\Longrightarrow}D$, if the following conditions are satisfied:

\begin{enumerate}
\item There exists a homotopy $h:[0,1]\times D\rightarrow E$ such that the
following hold true
\begin{eqnarray}
h(0,\cdot)  &  =f(\cdot),\nonumber\\
h([0,1]\times D^{-})\cap D  &  =\emptyset,\label{eq:topological-expansion}\\
h([0,1]\times D)\cap D^{+}  &  =\emptyset. \label{eq:topological-contraction}%
\end{eqnarray}

\item There exists a continuous map $\eta:\Lambda\rightarrow\Lambda$ for which%
\begin{equation}
\deg_{2}\left(  \eta\right)  \neq0, \label{eq:deg-covering}%
\end{equation}
moreover,\medskip

\begin{enumerate}
\item[a.] If $u>0$, then for any $\theta\in\Lambda$ there exists a linear map
$A_{\theta}:E_{\theta}^{u}\rightarrow E_{\eta\left(  \theta\right)  }^{u}$
such that $A_{\theta}(\partial B_{\theta}^{u})\subset E_{\eta\left(
\theta\right)  }^{u}\setminus B_{\eta\left(  \theta\right)  }^{u}$
($A_{\theta}$ is \emph{expanding}) and
\[
h(1,\left(  \theta;x,y\right)  )=(\eta\left(  \theta\right)  ;A_{\theta
}x,0)\in E_{\eta\left(  \theta\right)  }.
\]

\item[b.] If $u=0$, then
\[
h_{\theta}(1,\left(  \theta;y\right)  )=\left(  \eta\left(  \theta\right)
;0\right)  \in E_{\eta\left(  \theta\right)  }=E_{\eta\left(  \theta\right)
}^{s}.
\]
(In the above line we have omitted $x$ from the notation $\left(
\theta;x,y\right)  $ since $E^{u}$ is of dimension zero.)
\end{enumerate}
\end{enumerate}
\end{definition}

\begin{figure}[ptb]
\begin{center}
\includegraphics[width=13.5cm]{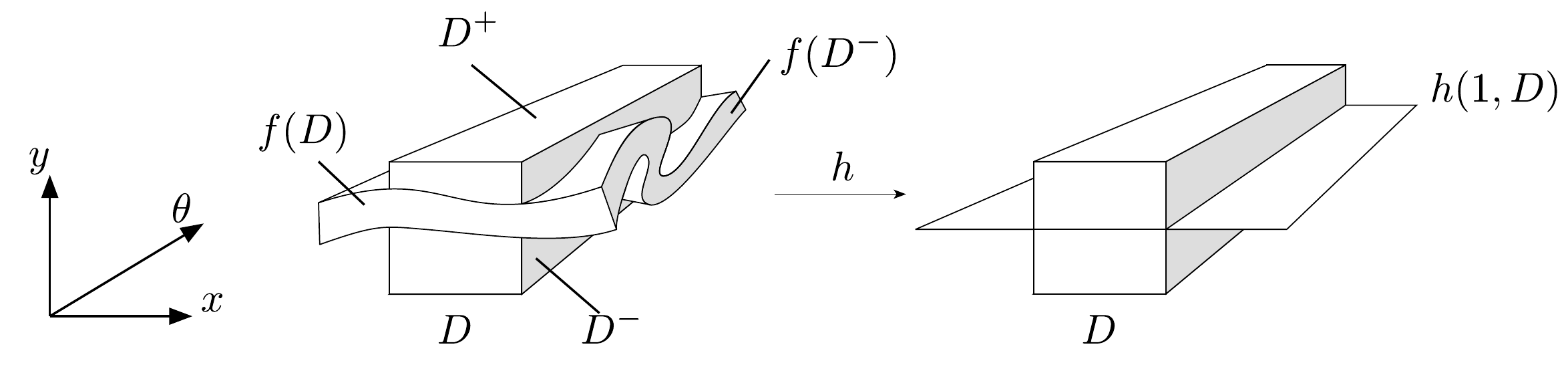}
\end{center}
\caption{Example of a homotopy from Definition \ref{def:covering} of the
covering $D\overset{f}{\Longrightarrow}D$.}%
\label{fig:covering}%
\end{figure}The intuition behind Definition \ref{def:covering} is similar to
what we discussed for Definition \ref{def:covering-theta}. We need to have
topological expansion in $x$ and topological contraction in $y$. In addition,
we assume that the dynamics on $\theta$ is homotopic to some map with nonzero
degree. Such property is visualized in Figure \ref{fig:covering}.

We make a couple of remarks before we formulate our second main result.

\begin{remark}
Condition $D\overset{f}{\Longrightarrow}D$ implies that $D_{\theta}%
\overset{f}{\Longrightarrow}D$ holds for any $\theta\in\Lambda$. (This follows by
taking $h_{\theta}=h|_{\left[  0,1\right]  \times D_{\theta}}$ and
$\Theta=\eta\left(  \theta\right)  $.) The implication in the other direction
is not always true. For instance, when $\pi\circ f(D)=\Theta$, meaning that
$f$ maps to a single fiber $E_{\Theta}$, then we can have $D_{\theta
}\overset{f}{\Longrightarrow}D$ for any $\theta\in\Lambda$, but $\eta$ satisfying
(\ref{eq:deg-covering}) will not exist.
\end{remark}

\begin{remark}
Condition (\ref{eq:deg-covering}) is quite natural for stroboscopic
(time-shift) maps of flows. In such setting, if the time shift along the flow
used to define the map is small enough, then it is possible to find a homotopy
to $\eta$ chosen to be the identity on $\Lambda$. Condition (\ref{eq:deg-covering}) is also automatically fulfilled in the setting of normal hyperbolicity; we show this in Lemma \ref{lem:nhim-covers}.
\end{remark}

\begin{remark}
\label{rem:generalization}
In (\ref{eq:deg-covering}) we use the degree modulo two of a map. This is
because we do not wish to impose any orientablity assumptions. If the considered manifolds
$\Lambda$ and $E$ are oriantable, one could use the Brouwer degree instead.
Condition (\ref{eq:deg-covering}) can also be replaced by requiring that the
degree computed at every point in $\Lambda$ is nonzero (Brouwer degree
computed at every point in $\Lambda$ is nonzero, if $\Lambda$ and $E$ are
oriantable); for which we do not need $\Lambda$ to be connected. (These
generalizations are highlighted in the footnote on page \pageref{footnote}
during the proof of Theorem \ref{th:invariance}.)
\end{remark}

We now formulate our second main result:

\begin{theorem}
\label{th:invariance}If $f:D\rightarrow E$ is a continuous mapping and
$D\overset{f}{\Longrightarrow}D$, then for every $\theta\in\Lambda$ there exists an
orbit in $D$ passing through $D_{\theta}$, i.e., there \correction{comment 5}{exists a sequence }$\left\{  v_{i}\right\}  _{i\in\mathbb{Z}}\subset D$, such that $v_{0}\in
D_{\theta}$ and $f\left(  v_{i}\right)  =v_{i+1}$, for all $i\in\mathbb{Z}$.
\end{theorem}

The proof is given in section \ref{proof:invariance}.
\begin{remark}\label{rem:sequences} In Theorems \ref{th:forward-invariance} and \ref{th:invariance} we obtain sets of points that remain in $D$ for iterates of the single map $f$. We can in fact just as well compose sequences of maps. 

To be precise, let $f_i : E\to E$ be a sequence of continuous maps and consider a dynamical system
\begin{equation}
v_{i+1}=f_i(v_i). \label{eq:nanaut-dynsys}
\end{equation}
Using mirror arguments to those used for the proof of Theorem \ref{th:forward-invariance} we can obtain forward trajectories of (\ref{eq:nanaut-dynsys})
in $D$ as long as $D_{\theta}\overset{f_i}{\Longrightarrow}D$ for all $i\in \mathbb{N}$ and all $\theta \in \Lambda$.

Similarly, if $D \overset{f_i}{\Longrightarrow}D$ for all $i\in \mathbb{Z}$, then using mirror arguments to those used for the proof of Theorem \ref{th:invariance} we can obtain full trajectories of (\ref{eq:nanaut-dynsys}) in $D$.  

The minor modifications needed for these results are highlighted in the footnotes during the course of the proofs of Theorems \ref{th:forward-invariance} and \ref{th:invariance} on pages \pageref{foot:th1} and \pageref{foot:th2}, respectively.
\end{remark}

We also have the following continuation results for continuous families of maps, which satisfy the
covering condition.

\begin{theorem}
\label{th:forward-continuation}Assume that we have a family of maps
$f_{\alpha}:D\rightarrow E$, which depends continuously on $\alpha\in\left[
0,1\right]  $. If for all $\alpha\in\left[  0,1\right]  $ and all $\theta
\in\Lambda,$ $D_{\theta}\overset{f_{\alpha}}{\Longrightarrow}D$, then for any
$\theta\in\Lambda$ there exists a compact connected component $C$ of $\left[
0,1\right]  \times D_{\theta}$ which meets both $\left\{  0\right\}  \times
D_{\theta}$ and $\left\{  1\right\}  \times D_{\theta},$ such that for any
$\left(  \alpha,v\right)  \in C$%
\[
f_{\alpha}^{n}\left(  v\right)  \in D\text{ for all }n\in\mathbb{N}.
\]

\end{theorem}

The proof is given in section \ref{proof:forward-continuation}.

\begin{theorem}
\label{th:continuation}Assume that we have a family of maps $f_{\alpha
}:D\rightarrow E$, which depends continuously on $\alpha\in\left[  0,1\right]
$. If for all $\alpha\in\left[  0,1\right]  $, $D\overset{f_{\alpha}}%
{\Longrightarrow}D$, then for any $\theta\in\Lambda$ there exists a compact connected
component $C$ of $\left[  0,1\right]  \times D_{\theta}$ which meets both
$\left\{  0\right\}  \times D_{\theta}$ and $\left\{  1\right\}  \times
D_{\theta},$ such that for any $\left(  \alpha,v\right)  \in C$ there exists
an orbit of $f_{\alpha}$ in $D$ passing through $v$, i.e., there exists as
sequence $\left\{  v_{i}\right\}  _{i\in\mathbb{Z}}\subset D$, such that
$v_{0}=v\in D_{\theta}$ and $f_{\alpha}\left(  v_{i}\right)  =v_{i+1}$, for
all $i\in\mathbb{Z}$.
\end{theorem}

The proof is given in section \ref{proof:continuation}.

\begin{figure}[ptb]
\begin{center}
\includegraphics[height=2.5cm]{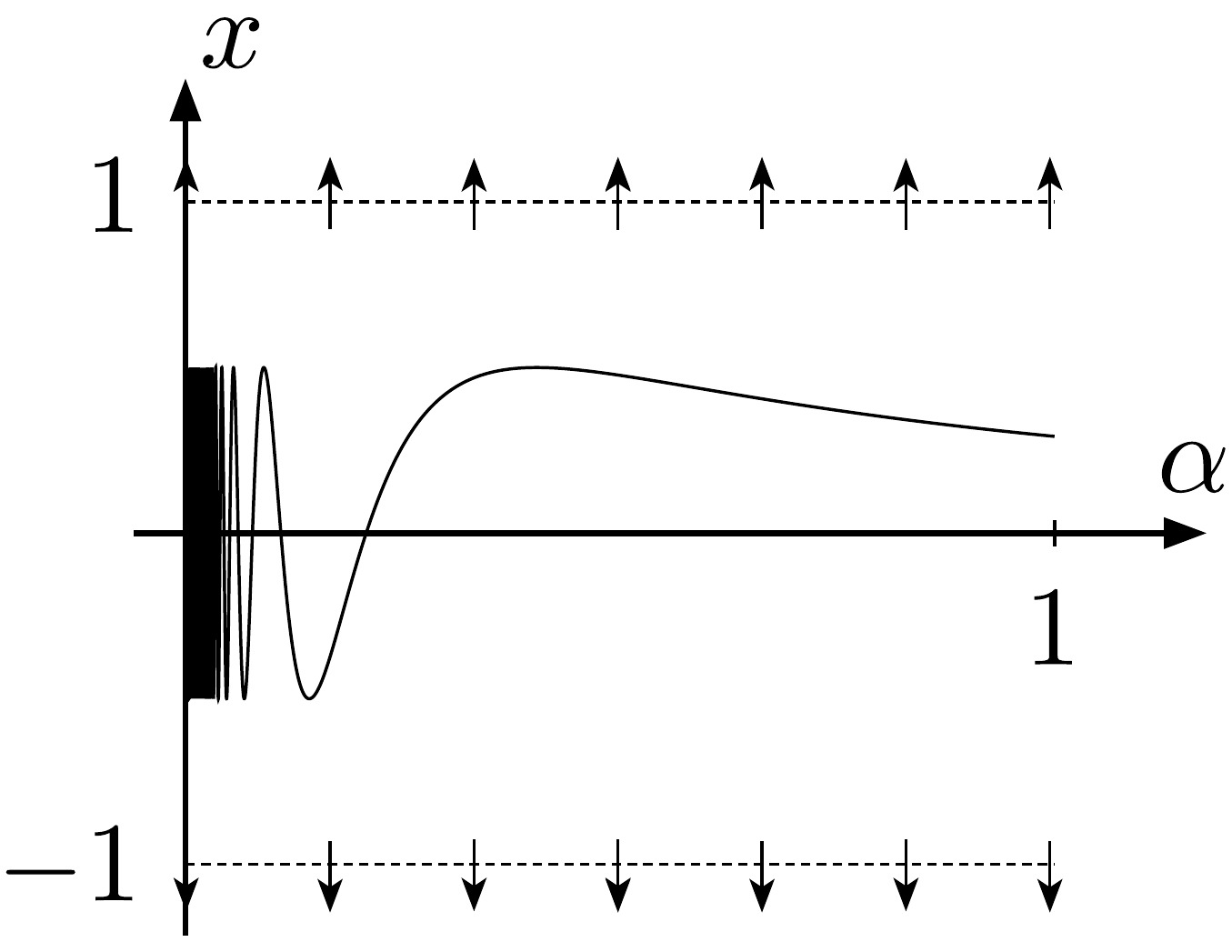}
\end{center}
\caption{The connected component $C$ from Remark \ref{rem:topol-sin}.}%
\label{fig:topol-sin}%
\end{figure}

\begin{remark}
\label{rem:topol-sin}We can not assert that $C$ from Theorems
\ref{th:forward-continuation}, \ref{th:continuation} is path connected. We can
see this if we take $f_{\alpha}:\mathbb{R\rightarrow R}$ with $f_{\alpha
}(x)=x+\alpha\left(  x-\frac{1}{2}\sin\frac{1}{\alpha}\right)  +g\left(
x\right)  ,$ where $g:\mathbb{R\rightarrow R}$ is continuous with $g|_{\left[
-1/2,1/2\right]  }=0$, $g|_{\left\{  x>1/2\right\}  }>0$ and $g|_{\left\{
x<-1/2\right\}  }<0$ (see Figure \ref{fig:topol-sin}). For $\alpha\in(0,1]$ we
have a family of hyperbolic fixed points $x_{\alpha}:=\frac{1}{2}\sin\frac
{1}{\alpha}$ of $f_{\alpha}$. Assumptions of Theorems
\ref{th:forward-continuation}, \ref{th:continuation} hold, since for all
$\alpha\in\left[  0,1\right]  $ the maps $f_{\alpha}$ stretch the interval
$D=\left[  -1,1\right]  $. We can not find a path connecting the set of fixed
points $\left[  -\frac{1}{2},\frac{1}{2}\right]  $ of $f_{0}$ with $x_{\alpha
}$. Nevertheless, we see that $C=\left\{  0\right\}  \times\left[  -\frac
{1}{2},\frac{1}{2}\right]  \cup\left\{  \left(  \alpha,x_{\alpha}\right)
|\,\alpha\in(0,1]\right\}  $ is connected (but not path connected).
\end{remark}

\begin{remark}
In the definition of the set $D$ we fixed the norms to be less than or equal
to one. This does not make it less general, since our results will hold in any
setting that is homeomorphic to the above.
\end{remark}

\subsection{Application in the context of normal hyperbolicity.\label{sec:norm-hyp}}

Below we give a corollary to bridge our results with the theory of NHIMs.
Before we proceed, we briefly recall the definition. 

\begin{definition}
Let $M$ be a smooth manifold and $f:M\rightarrow M$ a diffeomorphism. A
manifold $\Lambda\subset M$, invariant under $f$, i.e., $f\left(  \Lambda
\right)  =\Lambda$, is said to be \emph{normally hyperbolic} if there exist
a constant $C>0$, rates $0<\lambda<\mu^{-1}<1$ and a splitting {
\begin{equation}
T_{\Lambda}M=T\Lambda\oplus E^{u}\oplus E^{s},\label{eq:splitting}%
\end{equation}
which is }invariant under the action of the differential $df$ and {such that
for }$\theta\in\Lambda$%
\begin{eqnarray}
x  \in E_{\theta}^{u}&\iff\left\Vert df^{i}\left(
\theta\right)  x\right\Vert \leq C\lambda^{\left\vert i\right\vert }\left\Vert
x\right\Vert ,\qquad i\leq0,\label{eq:expansion-condition}\\
y \in E_{\theta}^{s}&\iff\left\Vert df^{i}\left(
\theta\right)  y\right\Vert \leq C\lambda^{i}\left\Vert y\right\Vert ,\qquad
i\geq0,\label{eq:contraction-condition}\\
w \in T_{\theta}\Lambda&\iff\left\Vert df^{i}\left(
\theta\right)  w\right\Vert \leq C\mu^{\left\vert i\right\vert }\left\Vert
w\right\Vert ,\qquad i\in\mathbb{Z}.\label{eq:central-condition}%
\end{eqnarray}

\end{definition}

We have the following lemma which states that normal hyperbolicity implies the
covering condition.

\begin{lemma}
\label{lem:nhim-covers}Let $\Lambda$ be a compact normally hyperbolic
invariant manifold for a diffeomorphism $f:M\rightarrow
M$, and let $k\in\mathbb{N}$ satisfy $k>\log_{\lambda}\frac{1}{C}%
$. Then there exists a neighborhood $D$ of $\Lambda$ such that
$D\overset{f^k}{\Longrightarrow}D.$
\end{lemma}

The proof is given in section \ref{sec:nhim-covers}.

From Theorem \ref{th:invariance} and Lemma \ref{lem:nhim-covers} we obtain the
following corollary.

\begin{corollary}
\label{cor:nhim-persists}Assume that $\Lambda$ is a compact normally
hyperbolic invariant manifold for a diffeomorphism $f:M\rightarrow M$, and assume that $D$ is such that $D\overset{f^k}{\Longrightarrow}D$. Let
$f_{\alpha}:M\rightarrow M$ be a family of
continuous maps, which depends continuously on $\alpha$. Assume that
$f_{0}=f$. Then for $\alpha$ for which
\begin{equation}
D\overset{f^k_{\alpha}}{\Longrightarrow}D,\label{eq:pert-cond}%
\end{equation}
$\Lambda$ persists as an invariant set of $f_{\alpha}$. Moreover, this set
projects surjectively onto $\Lambda$.
\end{corollary}

Corollary \ref{cor:nhim-persists} ensures that for a perturbation of $f$ the
NHIM will persist as an invariant set. In \cite{Floer} Floer proved a
similar result. He has shown that if $f_{\alpha}$ are homeomorphisms which
are close enough to $f$, then the NHIM persists along with its cohomology
ring. The first difference between our result and Floer's is that Corollary
\ref{cor:nhim-persists} provides a verifiable condition (\ref{eq:pert-cond})
for the persistence of the NHIM, effectively getting rid of the `close enough'
part of the Floer's statement. (For small $\lambda$ (\ref{eq:pert-cond}) will
hold, and for a particular system we can explicitly check for which $\lambda$
(\ref{eq:pert-cond}) will be satisfied.) In our setting, the existence of the
NHIM is in fact not even necessary, since (\ref{eq:pert-cond}) alone
establishes the existence of the invariant sets. Another difference is that in
Corollary \ref{cor:nhim-persists} it is enough that $f_{\alpha}$ are
continuous; we do not need them to be homeomorphisms as is required in
\cite{Floer}. Floer proves that the cohomology ring of the invariant set which persists contains the
cohomology ring of the original manifold as a subring. We prove that the
topology of the original manifold is in a sense `preserved', but in our
statement this is expressed by the fact that the invariant set which persists
projects surjectively onto the original NHIM. Moreover, we know by Theorem
\ref{th:continuation} that the invariant manifold `continues' in the sense
that the sets for different parameters are linked on each fiber by a compact
connected component.

A desirable feature of our result is that the covering condition
(\ref{eq:pert-cond}) can be checked using computer assisted techniques, which
makes our results applicable in practice. 

From Remark \ref{rem:sequences} and Lemma \ref{lem:nhim-covers} we also obtain the following result for random perturbations of NHIMs. (See \cite{Bates} for a similar persistence result of NHIMs for random perturbations of flows.)
\begin{corollary}
\label{cor:random} Let $\Lambda$ be a compact normally hyperbolic invariant
manifold for a diffeomorphism $f:M\rightarrow M$, and let $k\in\mathbb{N}$
satisfy $k>\log_{\lambda}\frac{1}{C}$. Assume that $D$ is such that
$
D\overset{f^k}{\Longrightarrow}D.
$
Let $(\Omega,\mathcal{F},\mathbb{P})$ be a probability space, let 
$T:\Omega\rightarrow\Omega$ and let $\phi:\mathbb{Z}\times\Omega\times
M\rightarrow M$ be a random dynamical system over $T$, i.e. $\phi\left(
0,\omega\right)  =Id$ and
\[
\phi(n+m,\omega)=\phi(n,T^{m}(\omega))\phi\left(  m,\omega\right)  .
\]
If $\phi$ is close enough to $f$ so that for any $\omega\in\Omega$, 
$
D\overset{\phi(k,\omega)}{\Longrightarrow}D,
$
then the NHIM persists as a set of trajectories of
$\phi$. Moreover, the set projects surjectively onto $\Lambda$.
\end{corollary}


\section{Examples of application.\label{sec:example}}
The perturbations of a system with a NHIM can be
\correction{comment 6}{such that }the perturbed maps are no longer normally hyperbolic, but we can still apply our results. Below we give an example
of such a system. 
\subsection{Toy example\label{sec:toy}}
\begin{figure}[ptb]
\begin{center}
\includegraphics[height=4.5cm]{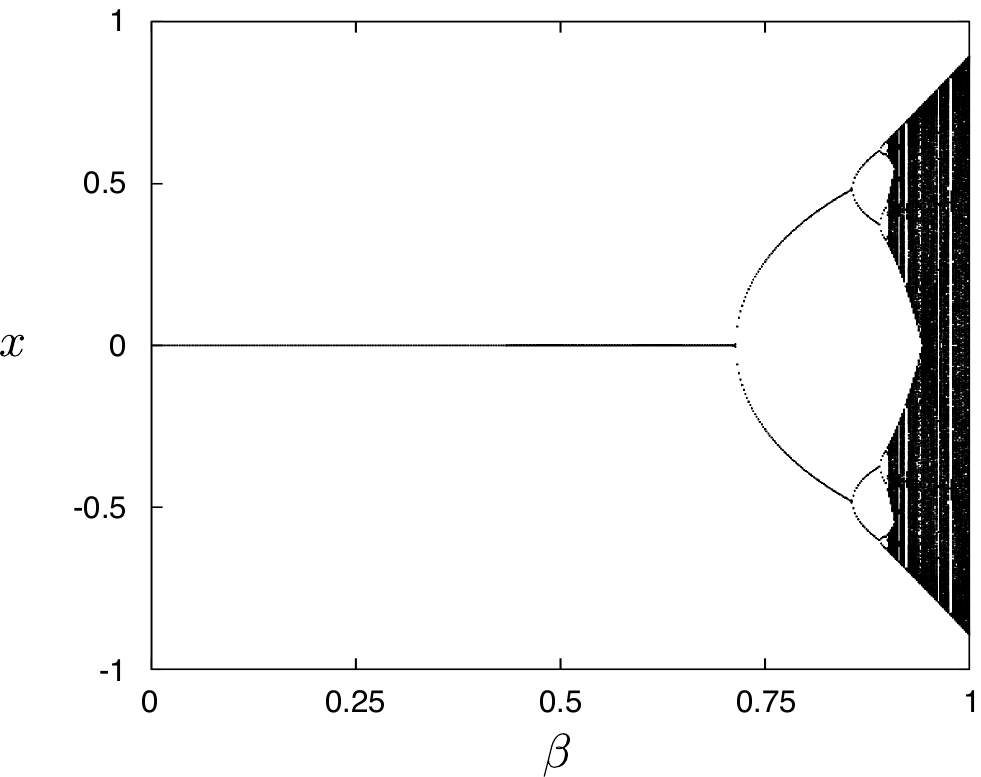}
\end{center}
\caption{The projection onto the $x$ coordinate of the invariant set from the
example from section \ref{sec:example}, depending on the parameter $\beta$. }%
\label{fig:logist}%
\end{figure}

We start with an example where the dynamics on the unstable coordinate is decoupled from the rest of the coordinates. The aim is to provide a simple model on which we can demonstrate some features, without having to engage in computations.

Let $\Lambda$ be a one dimensional circle, parameterized by $\theta\in
\lbrack0,2\pi)$. Let $E^{u}$ be a trivial bundle over $\Lambda$ (i.e.,
$E^{u}=\Lambda\times\mathbb{R}$), let $E^{s}$ be a M\"{o}bius bundle over
$\Lambda$, and let $E=E^{u}\oplus E^{s}$. Take $\mu\in\mathbb{R}$,
$|\mu|<\frac{1}{2}$, and two maps $f_{0},f_{1}:E\rightarrow E$ defined as
\begin{eqnarray}
f_{0}\left(  \theta;x,y\right)   &  =\left(  3\theta\,\mathrm{mod}%
\,2\pi;4x,\mu y\right)  ,\nonumber\\
f_{1}\left(  \theta;x,y\right)   &  =\left(  3\theta\,\mathrm{mod}%
\,2\pi;-3x+5x^{3},\frac{1}{2}\sin\theta+\mu y\right)  . \label{eq:f-example}%
\end{eqnarray}
The maps $f_{0}$ and $f_{1}$ expand the M\"{o}bius strip along $\theta$,
wrapping it around itself three times, and squeeze it along the $y$ coordinate
(see Figure \ref{fig:mobius}). On the $x$ coordinate we have decoupled dynamics.

\begin{figure}[ptb]
\begin{center}
\includegraphics[height=4.5cm]{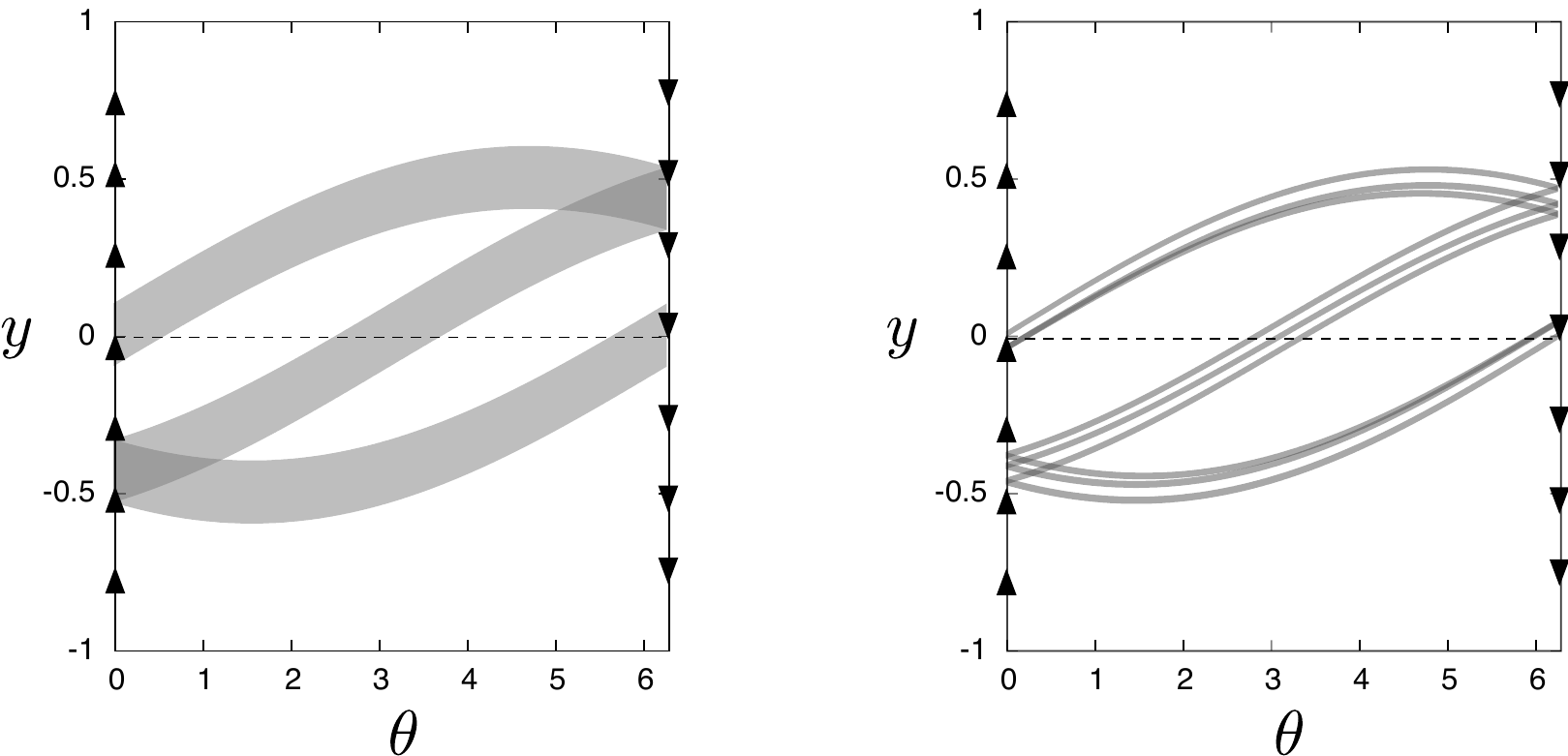}
\end{center}
\caption{The image of $f_{1}(D)$ (left) and $(f_{1})^{2}(D)$ (right), for the
map $f_{1}$ defined in (\ref{eq:f-example}) from the example from section
\ref{sec:example}, projected onto the M\"{o}bius bundle $E^{s}$. On the plot,
the vertical edge $\{\theta=0\}$ is identified with the vertical edge
$\{\theta=2\pi\}$ with reversed sign (which is indicated by arrows on the
plots). Here we took $\mu=\frac{1}{10}$ in (\ref{eq:f-example}). }%
\label{fig:mobius}%
\end{figure}In this example we will discuss invariant sets for a family of
maps $f_{\beta}:E\rightarrow E$, defined as
\[
f_{\beta}:=\left(  1-\beta\right)  f_{0}+\beta f_{1},\qquad\text{for }\beta
\in\left[  0,1\right]  .
\]
For $\beta=0$ the set $\left\{  x=y=0\right\}  $ is invariant, and on it the
rate conditions hold; i.e. the dynamics in the hyperbolic directions $x,y$ is
stronger than on $\theta$. As we increase $\beta$, the expansion along $x$
becomes weaker than the expansion along $\theta$. This means that the
classical tools can not ensure that the manifold survives. If we take though
$D=\left\{  \left(  \theta;x,y\right) |\;\left\vert x\right\vert
\leq1,\ \left\vert y\right\vert \leq1\right\}  $, fix $\beta\in\lbrack0,1]$,
consider a homotopy
\[
h\left(  \alpha,\left(  \theta;x,y\right)  \right)  =\left(  1-\alpha\right)
f_{\beta}\left(  \theta;x,y\right)  +\alpha\left(  3\theta\,\mathrm{mod}%
\,2\pi;2x,0\right)  ,
\]
and $\eta(\theta)=3\theta\, \mathrm{mod}\, 2\pi$, then it is a simple exercise
to verify that%
\begin{equation}
D\overset{f_{\beta}}{\Longrightarrow}D. \label{eq:example-covering}%
\end{equation}
The reason why (\ref{eq:example-covering}) holds boils down to the fact that
on the $y$ coordinate we have contraction and the cubic terms on the
coordinate $x$ ensure expansion away from zero. Since $\eta\left(
\theta\right)  =3\theta\,\mathrm{mod}\,2\pi$, we see that $\deg
_{2}\left(  \eta\right)  =1$.

Theorem \ref{th:invariance} ensures that for any $\beta\in\left[  0,1\right]
$ there is an invariant set in $D$, with trajectories in $D$ passing through
each $\theta\in\lbrack0,2\pi)$. Theorem \ref{th:invariance} does not claim
that the invariant set is a manifold. In fact it is not a manifold, which we
can see if we look at the projections in Figures \ref{fig:logist} and
\ref{fig:mobius}. Figure \ref{fig:logist} contains the plot of the invariant set of
$x\mapsto\pi^{u}\circ f_{\beta}\left(  \theta;x,y\right)  $ for $\beta
\in\left[  0,1\right]  $. (The dynamics of $f_{\beta}$ on $x$ is decoupled
from other variables, so the set is independent from the choice of $\theta
,y$.) We see that for $\beta$ close to 1 our set will be chaotic. This is
because the function passes through logistic type bifurcations as we increase
$\beta$. In Figure \ref{fig:mobius} we take the parameter $\mu=\frac{1}{10}$,
fix $\beta=1$ and plot the projections of $f_{\beta}\left(  D\right)  $ and
$f_{\beta}^{2}\left(  D\right)  $ onto the M\"{o}bius strip. We see that if we
were to consider $f_{\beta}^{k}\left(  D\right)  $ for higher $k$, then we
would see the emergence of a Cantor structure of our invariant set. Theorem
\ref{th:continuation} states that the resulting invariant set for different
$\beta$ `continues' as the parameter changes, which we see is the case in
our example.

The main feature of this example is that we have started with a manifold which
satisfied the rate conditions, and perturbed the system into the parameter
range where the rate conditions fail. Nevertheless, our method establishes the
existence of an invariant set for all parameters.

In our example the dynamics on $x$ is decoupled from the dynamics on the
M\"{o}bius bundle. We have done this for simplicity. The assumptions of
Theorems \ref{th:invariance}, \ref{th:continuation} are robust under small
perturbations, so we will also obtain the results for any map that is
appropriately close to $f_{\beta}$, for one of the $\beta\in\left[
0,1\right]  .$

In our example we were able to verify (\ref{eq:deg-covering}) because
$f_{\beta}$ on coordinate $\theta$ were given as $3\theta\,\mathrm{mod}%
\,2\pi$. If we were to take $k\theta\,\mathrm{mod}\,2\pi$ with an even
number $k$, then we would get $\deg\left(  \eta\right)  =0$, and we would not
be able to apply Theorem \ref{th:invariance}. We finish by observing that in
such a setting we can still use Theorem \ref{th:forward-invariance} to obtain an
invariant set of points that stay in $D$ for all (forward) iterations.

The above was just a toy example. Similar features though can be found for
instance in the Kuznetsov system (see \cite{Kuzn,Wilczak}), where we have a
hyperbolic invariant set in $\mathbb{R}^{3}$, which has a Cantor set structure.

\begin{figure}[ptb]
\begin{center}
\includegraphics[height=5cm]{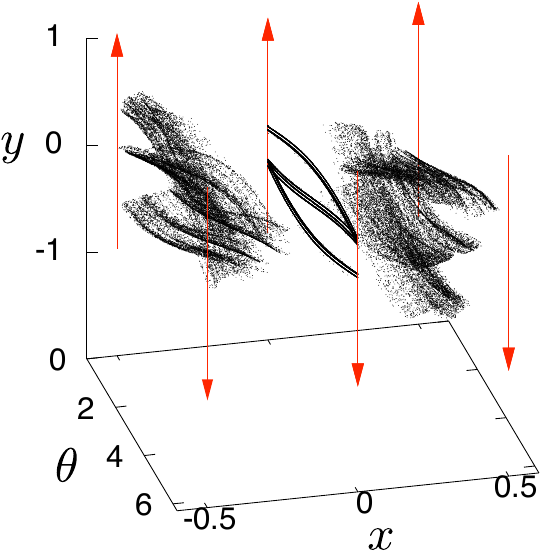}
\end{center}
\caption{The invariant set of (\ref{eq:CAP-map}).}%
\label{fig:CAP}%
\end{figure}
\correction{comment 3}{
\subsection{An example with a computer assisted proof\label{sec:CAP}}
Here we modify our example from section \ref{eq:CAP-map} by coupling the dynamics between the
coordinates. We consider the following map}
\correction{}{
\begin{eqnarray}
& f\left(  \theta;x,y\right)  = \label{eq:CAP-map}  \\ 
& \left(  3\theta+xy\sin(\theta)\text{
}\mathrm{mod}\text{ }2\pi;4x^{3}-\frac{8}{5}x+\frac{1}{2}xy,\mu
y+\frac{2}{5}\sin\theta+x\cos(\theta)\right)  , \nonumber
\end{eqnarray}
with $\mu=1/10.$ This map results from taking $f_{\beta}$ from the previous
section with $\beta=0.8$, and by adding the coupling terms $xy\sin(\theta)$,
$\frac{1}{2}xy$ and $x\cos(\theta)$ to the $\theta$-, $x$- and $y$-coordinates, respectively. The choice of such coupling was to a large extent arbitrary. We
wanted a nontrivial but simple example, with some interesting features.}

\correction{}{In Figure \ref{fig:CAP} we give the plot of a numerically obtained representation of the
invariant set we will establish. On the plot, the front face $\theta=2\pi$ is
identified with the back face $\theta=0$, but should be glued together
according to the arrows to take into account the fact that $E_{s}$ is a Mobius bundle. }

\correction{comment 1}{For $x=0$ we see the invariant set from our previous example (compare 
Figures \ref{fig:mobius} and \ref{fig:CAP}); we have intentionally chosen our coupling to
preserve it. The coupling is strong enough to distort the two attracting fixed points
of the uncoupled map $f_{\beta=0.8}$ on $x$ (See Figure \ref{fig:logist}) to become the two `chaotic clouds' from Figure \ref{fig:CAP}.}

\correction{}{For this example we provide a computer assisted proof that for $D=\left\{
\left(  \theta;x,y\right)  :\left\Vert x\right\Vert \leq1,\left\Vert
y\right\Vert \leq1.2\right\}  $ we will have $D\overset{f}{\Longrightarrow}D$,
which, by Theorem \ref{th:invariance}, implies the existence of an invariant set
in $D$. This is done by considering the following homotopy 
\begin{eqnarray}
h\left(  \alpha,\left(  \theta;x,y\right)  \right)   :=&(3\theta+\left(
1-\alpha\right)  xy\sin(\theta)\text{ }\mathrm{mod}\text{ }2\pi; \label{eq:homotopy-CAP} \\
& \alpha2x+\left(  1-\alpha\right)  \left(  4x^{3}-\frac{8}{5}x+\frac
{1}{2}x  y\right), \nonumber \\
& \left(  1-\alpha\right)  \left(  \mu y+\frac{2}{5}\sin\theta
+x\cos(\theta)\right)  ). \nonumber
\end{eqnarray}}

\correction{}{Condition (\ref{eq:deg-covering}) follows directly from the definition of $h$. We
validate (\ref{eq:topological-expansion}--\ref{eq:topological-contraction}) by using interval arithmetic. Interval arithmetic involves enclosing
numbers in intervals that account for possible round-off errors, and
performing arithmetic operation on these intervals. The output of these
operations are intervals, which account for the numerical error and enclose
the true result. }

\correction{}{We give the full code which we have used for our computer
assisted proof in \ref{sec:app-code} and follow with a number of comments associated to the particular routines.
The validation of (\ref{eq:topological-expansion}%
--\ref{eq:topological-contraction}) is based on subdividing the domains into
small sets and checking the correct topological alignment by means of
inequalities between intervals. A sample of such bounds obtained by our computer program is
depicted  in Figure \ref{fig:CAP-covering}.}

\begin{figure}[ptb]
\begin{center}
\includegraphics[height=3cm]{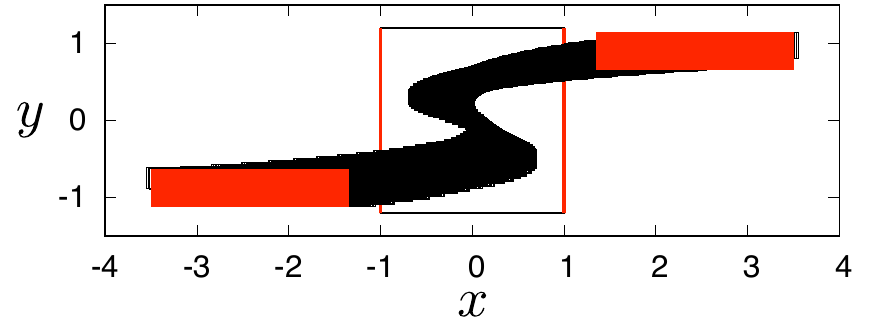}
\end{center}
\caption{A sample of computations performed by the computer. We subdivide $D$ into small cubes, and check inequalities which ensure that local projections onto $x,y$ 
do not intersect with $D^+$. We also compute images of cubes from $D^-$ (in red) and check that they map to the left and to the right of the set $D$.}%
\label{fig:CAP-covering}%
\end{figure}

\correction{comment 2}{
\subsection{Finding invariant sets and covering relations\label{sec:finding-sets}}
If the system under consideration possesses a NHIM, then it is a natural choice to position $D$ around the NHIM, aligning $D^+$ and $D^-$ with the stable and unstable bundles, respectively. When the perturbation is far from the normally hyperbolic case, or if we want to apply our methods in a setting where no NHIM exists, we can use the following numerical method.
}

\correction{}{We can select some domain within which we expect to find our invariant object, and subdivide it into cubes. Then, using interval arithmetic we can propagate such cubes and discard those that will leave the domain after some iterate of the map. Those cubes that do not escape are dissected into smaller cubes, and the procedure can be repeated. If some invariant set is within our domain, it will be detected by this method.}

\correction{}{The reason for using interval arithmetic is that even if just a single point from a considered cube is an element of the invariant object, then it will not leave the domain, and the cube will not be discarded. (For instance, the discussed methodology works very well in the normally hyperbolic setting, to find an enclosure of the stable manifold.)}

\correction{}{The positioning of the enclosure that comes out of the algorithm can give an insight into how $D^+$ and $D^-$ should be positioned. In Figure \ref{fig:domains} we show an outcome of the procedure applied to the map (\ref{eq:CAP-map}) for domains of the form $\{(\theta;x,y): x^2+y^2<r^2\}$, for $r=2$ and $r=1/2$. (On the plot we depict cubes with $\theta=\pi/3$, because we took a liking to the shape on the right hand side.) For $r=2$ we see that $D^+$ should be towards the vertical and $D^-$ towards the horizontal axis. For $r=1/2$ though we obtain an enclosure that does not give a clear indication how $D^+$ and $D^-$ could be positioned. Finding suitable $D$ in complicated systems is not an easy task and is likely to involve trial and error.  }

\begin{figure}[ptb]
\begin{center}
\includegraphics[height=5cm]{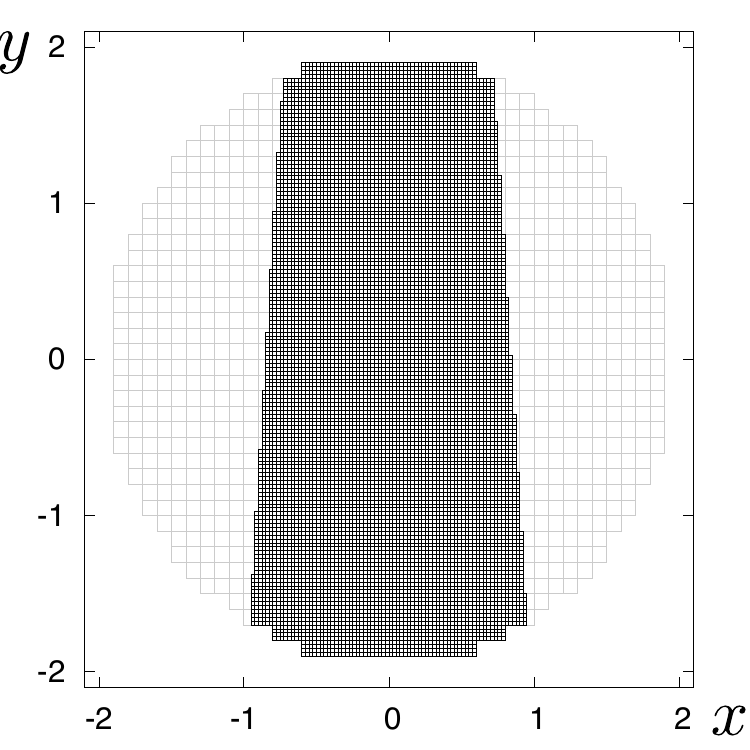}\hspace{1cm}\includegraphics[height=5cm]{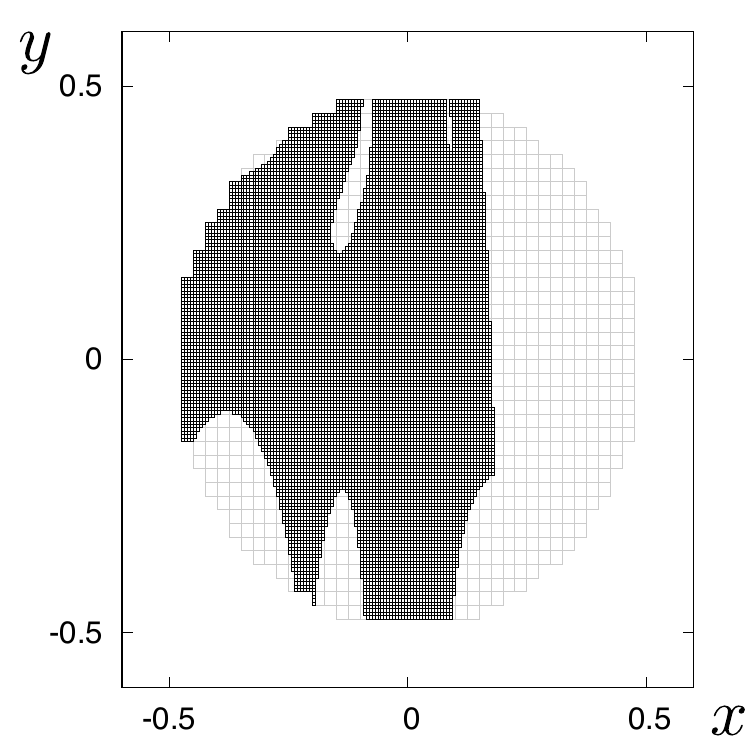}
\end{center}
\caption{In grey is the set of initial cubes containing $\theta=\pi/3$ at the start of our algorithm. The left plot is for $r=2$ and the right for $r=1/2$. The black cubes are those left after two steps of the procedure.\label{fig:domains}}%
\end{figure}

\section{Embedding into reals}

In this section we shall embed $E$ in $\mathbb{R}^{N}$.
We will then extend the map $f$ so that it is defined on a set with nonempty
interior in $\mathbb{R}^{N}$. \correction{comment 7}{Such embedding will be useful for us since to find two points $v_1, v_2\in D$ such that $v_2 = f(v_1)$ we will be able to do so more easily by embedding $f(v_1)$ and $v_2$ into $\mathbb{R}^{N}$, computing their difference, and solving for zero. Searching for zeros in $\mathbb{R}^{N}$ is more tractable than finding two points on a vector bundle that map one into the other.} 

The vector bundle $E$ is an $n$-dimensional smooth manifold, $n=c+u+s$. By the
Whitney embedding theorem \cite{Whitney}  there exists a smooth embedding $\omega
:E\rightarrow\mathbb{R}^{2n}$. Let $N_{w}\left(  \omega\left(  E\right)
\right)  \mathbf{\subset}\mathbb{R}^{2n}$ stand for the normal space to the
manifold $\omega\left(  E\right)  $ in $\mathbb{R}^{2n}$ at $w\in\omega\left(
E\right)  $. (Since $\omega\left(  E\right)  $ is a manifold of dimension $n$,
the dimension of $N_{w}\left(  \omega\left(  E\right)  \right)  $ is $n$.) We
consider the tubular neighborhood of $\omega\left(  E\right)  $
\begin{equation}
\mathcal{T}:=\left\{  w+z|\;w\in\omega\left(  E\right),\ z\in N_{w}\left(
\omega\left(  E\right)  \right),\ \left\Vert z\right\Vert _{\mathbb{R}^{2n}%
}\leq\delta\left(  w\right)  \right\}  \subset\mathbb{R}^{2n},
\label{eq:tubular-neighb}%
\end{equation}
where $\delta:\omega\left(  E\right)  \rightarrow\mathbb{R}_{+}$ is
continuous. Let us abuse the notation slightly by introducing a number
$\delta\in\mathbb{R}$ defined as
\begin{equation}
\delta:=\min_{w\in\omega\left(  D\right)  }\delta\left(  w\right)  .
\label{eq:delta-def}%
\end{equation}
Since $D$ is compact $\delta>0$ is well defined.

\begin{notation}
\label{notation2}For $v\in E$ we shall write a pair $\left(  v,z\right)  $ to
represent a point $\omega\left(  v\right)  +z\in\mathcal{T}$. In this
convention writing the pair $\left(  v,z\right)  $ implies that $z\in
N_{\omega\left(  v\right)  }\left(  \omega\left(  E\right)  \right)  .$ In the
same way by writing $\left(  \theta;x,y,z\right)  $ we mean the point
$\omega\left(  \theta;x,y\right)  +z\in\mathcal{T}$, and imply that $z\in
N_{\omega\left(  \theta;x,y\right)  }\left(  \omega\left(  E\right)  \right)
.$
\end{notation}

\begin{figure}[ptb]
\begin{center}
\includegraphics[width=10cm]{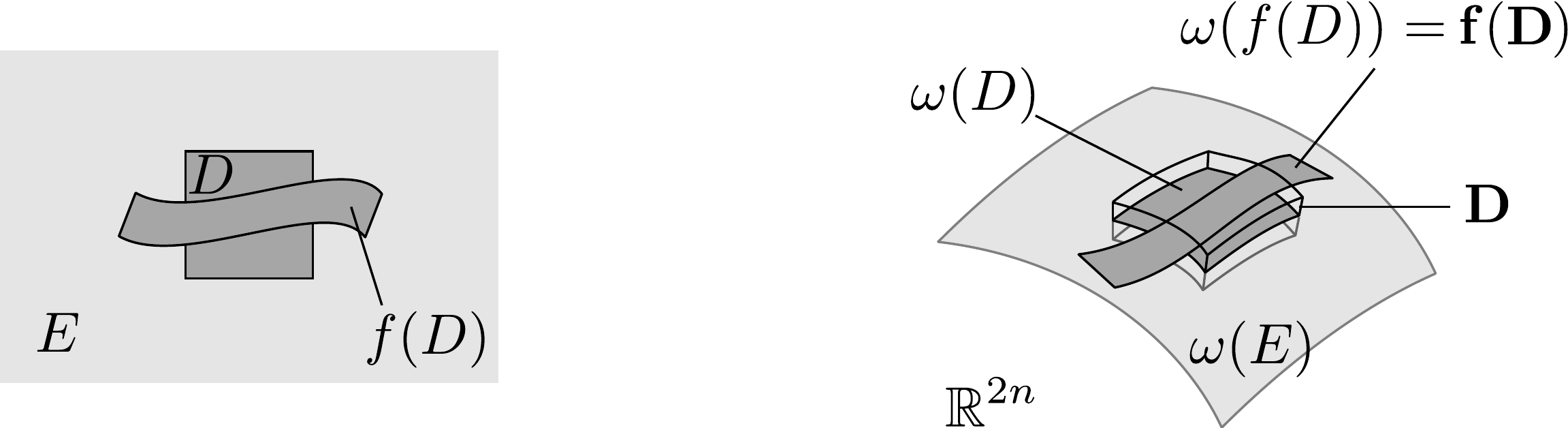}
\end{center}
\caption{Depiction of our embedding. The set $\mathbf{D}$ is a tubular
neighborhood of $\omega(D)$. The map $\mathbf{f}$ defined on $\mathbf{D}$
contracts in the normal direction onto $\omega(f(D))$. This means that the
normal direction is treated as an additional contracting coordinate. }%
\label{fig:embedding}%
\end{figure}

Using Notation \ref{notation2} we define the following subsets of
$\mathbb{R}^{2n}$
\begin{equation}
\mathbf{D}:=\left\{  \left(  v,z\right)  \in\mathcal{T}|\;v\in D,\ \left\Vert
z\right\Vert _{\mathbb{R}^{2n}}\leq\delta\right\}  , \label{eq:D-def}
\end{equation}
\[
\mathbf{D}_{\theta}:=\mathbf{D}\cap\left\{  \left(  v,z\right)  |\;\pi\left(
v\right)  =\theta\right\}  ,\qquad\mathbf{D}_{U}:=\bigcup_{\lambda\in
U}\mathbf{D}_{\lambda}.
\]
We define a map
\[
\mathbf{f}:\mathbf{D}\rightarrow\mathbb{R}^{2n},
\]
as
\begin{equation}
\mathbf{f}\left(  v,z\right)  :=\left(  f\left(  v\right)  ,0\right)  , \label{eq:bold-f-def}
\end{equation}
where the zero on the right hand side is on the $z$-coordinate. (In other words, $\mathbf{f}(
\omega\left(  v\right)  +z)  =\omega\left(  f\left(  v\right)  \right)
$; see Figure \ref{fig:embedding}.) If we need to include more detail, using
also the convention of Notation \ref{notation2}, we can write
\[
\mathbf{f}\left(  \theta;x,y,z\right)  =\left(  f\left(  \theta;x,y\right)
,0\right)  ,
\]
where the zero on the right hand side is on the $z$-coordinate. Observe that directly from the
definition of $\mathbf{f}$ we have the following:

\begin{lemma}
\label{lem:forward-traj-emb}\label{lem:emb-the-same}If for $\left(
v_{1},z_{1}\right)  ,\left(  v_{2},z_{2}\right)  \in\mathcal{T}$ we have
$\left(  v_{2},z_{2}\right)  =\mathbf{f}\left(  v_{1},z_{1}\right)  $, then
$v_{2}=f\left(  v_{1}\right)  $.
\end{lemma}

\begin{proof}
By definition, $\left(  v_{2},z_{2}\right)  =\mathbf{f}\left(  v_{1}%
,z_{1}\right)  $ implies $\omega\left(  f\left(  v_{1}\right)  \right)
=\omega\left(  v_{2}\right)  +z_{2}$. Since $\mathcal{T}$ is a tubular
neighborhood of $\omega\left(  E\right)  $, each point in $\mathcal{T}$ is
represented in a unique way as $\omega\left(  v\right)  +z$. This means that
$\omega\left(  f\left(  v_{1}\right)  \right)  =\omega\left(  v_{2}\right)
+z_{2}$ implies that $z_{2}=0$ and $\omega\left(  f\left(  v_{1}\right)
\right)  =\omega\left(  v_{2}\right)  ,$ which in turn gives $f\left(
v_{1}\right)  =v_{2}$, as required.
\end{proof}

\begin{remark}
\label{rem:why-embedding}When we looking for two points $v_{1},v_{2}\in D$
such that $v_{2}=f(v_{1})$, by Lemma \ref{lem:emb-the-same} we can 
solve
\begin{equation}
\mathbf{f}\left(  v_{1},z_{1}\right)  -\left(  v_{2},z_{2}\right)  =0,
\label{eq:zero-reason}%
\end{equation}
for $\left(  v_{1},z_{1}\right)  ,\left(  v_{2},z_{2}\right)  \in\mathbf{D}$.
(In (\ref{eq:zero-reason}) we are subtracting two vectors in $\mathbb{R}^{2n}%
$.) 
\end{remark}

\section{Proof of Theorem \ref{th:forward-invariance}%
\label{proof:forward-invariance}}

\begin{proof}
Let us fix $\theta=\Theta_{0}\in\Lambda$. Our objective will be to find a
trajectory starting from $D_{\Theta_{0}}$, which remains in $D$ for all
forward iterates. We start by finding trajectories of length $k$.

Let $0^{2n}$ denote zero in $\mathbb{R}^{2n}.$ For fixed $k\in\mathbb{N}$ we
consider the following sets (recall that $\mathbf{D}$ was defined in (\ref{eq:D-def}))
\begin{eqnarray}
\overline{X} &  :=D_{\Theta_{0}}^{u}\times\underset{k}{\underbrace
{\mathbf{D}\times\ldots\times\mathbf{D}}},\label{eq:X-forward-def}\\
Y &  :=\underset{k}{\underbrace{\mathbb{R}^{2n}\times\ldots\times
\mathbb{R}^{2n}}}\times E,\label{eq:Y-forward-def}\\
Z &  :=\underset{k}{\underbrace{\{0^{2n}\}\times\ldots\times\{0^{2n}\}}}\times\left\{
\left(  \theta;0,y\right)  |\;\theta\in\Lambda,\ \left\Vert y\right\Vert
_{s}<1\right\}  .\label{eq:Z-forward-def}%
\end{eqnarray}
We consider $\overline{X}$ as a subset of $E_{\Theta_{0}}^{u}\times
\mathcal{T}\times\ldots\times\mathcal{T}$, so%
\[
X=\mathrm{int}\overline{X}=\mathrm{int}D_{\Theta_{0}}^{u}\times\underset
{k}{\underbrace{\mathrm{int}\mathbf{D}\times\ldots\times\mathrm{int}%
\mathbf{D}}}.
\]
Note that since $\Lambda$ is compact so is $\overline{X}$. $Y$ is a manifold
without boundary and $Z$ is its submanifold, with%
\begin{eqnarray}
\overline{Z} &  =\{0^{2n}\}\times\ldots\times\{0^{2n}\}\times\left\{  \left(
\theta;0,y\right)  |\;\theta\in\Lambda,\ \left\Vert y\right\Vert _{s}%
\leq1\right\}  ,\label{eq:Z-1}\\
\partial Z &  =\{0^{2n}\}\times\ldots\times\{0^{2n}\}\times\left\{  \left(
\theta;0,y\right)  |\;\theta\in\Lambda,\ \left\Vert y\right\Vert _{s}=1\right\}
.\label{eq:Z-2}%
\end{eqnarray}

The manifolds $X$ and $Z$ are of complementary dimension with respect to $Y$:%
\[
\dim X=u+2kn,\qquad\dim Z=c+s,\qquad\dim Y=2kn+n.
\]

To show the existence of an orbit of length $k$ in $D$ we consider a map
\[
F=\left(  F_{1},\ldots,F_{k},F_{k+1}\right)  :\overline{X}\rightarrow Y,
\]
which is defined as follows. For
\begin{equation}
\mathbf{x}=\left(  \left(  \Theta_{0};x_{0}\right)  ,\left(  \theta_{1}%
;x_{1},y_{1},z_{1}\right)  ,\ldots,\left(  \theta_{k};x_{k},y_{k}%
,z_{k}\right)  \right)  \in\overline{X} \label{eq:x-form}%
\end{equation}
we define\label{foot:th1}\footnote{To obtain the generalization stated in Remark \ref{rem:sequences} here we should use $\mathbf{f}_i$ in the definition of $F_i(\mathbf{x})$ for $i=1,\ldots,k$ and $f_{k+1}$ in the definition of $F_{k+1}(\mathbf{x})$; throughout the reminder of the proof we would use homotopies resulting  from the coverings $D_{\theta}\overset{f_i}{\Longrightarrow}D$ in the respective places that follow.}
 (recall that $\mathbf{f}$ was defined in (\ref{eq:bold-f-def}))
\begin{eqnarray}
F_{1}\left(  \mathbf{x}\right)   &  :=\mathbf{f}\left(  \Theta_{0}%
;x_{0},0,0\right)  -\left(  \theta_{1};x_{1},y_{1},z_{1}\right)  ,\nonumber\\
F_{2}\left(  \mathbf{x}\right)   &  :=\mathbf{f}\left(  \theta_{1};x_{1}%
,y_{1},z_{1}\right)  -\left(  \theta_{2};x_{2},y_{2},z_{2}\right)
,\nonumber\\
&\;\;\;\vdots\label{eq:Fk-def}\\
F_{k}\left(  \mathbf{x}\right)   &  :=\mathbf{f}\left(  \theta_{k-1}%
;x_{k-1},y_{k-1},z_{k-1}\right)  -\left(  \theta_{k};x_{k},y_{k},z_{k}\right)
,\nonumber\\
F_{k+1}\left(  \mathbf{x}\right)   &  :=f\left(  \theta_{k};x_{k},y_{k}\right).
\nonumber
\end{eqnarray}
Our objective will be to prove that there exists an $\mathbf{x\in}\overline
{X},$ such that
\begin{equation}
F\left(  \mathbf{x}\right)  \in Z. \label{eq:Fx_in_Z}%
\end{equation}
Observe that by Lemma \ref{lem:forward-traj-emb}, (\ref{eq:Fx_in_Z})
establishes the existence of a trajectory of $f:D\rightarrow E,$ that starts
in $v=\left(  \Theta_{0};x_{0},0\right)  $ and remains in $D$ for $k$ iterates
of $f$:%
\[
f^{i}\left(  v\right)  \in D\qquad\text{for }i=1,\ldots,k.
\]

Our plan is to establish (\ref{eq:Fx_in_Z}) by showing that the intersection
number $I_{2}\left(  F,X,Z\right)  =1$; then (\ref{eq:Fx_in_Z}) will follow
from the intersection property.

The first thing to show is that $F$ is admissible (in the sense of Definition
\ref{def:admissible}). We shall consider an $\mathbf{x}$ of the form
(\ref{eq:x-form}) which lies in the boundary $\partial X$ and show that
$F\left(  \mathbf{x}\right)  \notin\overline{Z}$. There are several
possibilities how $\mathbf{x}$ can lie on $\partial X$, which we will consider one
by one below. In the following argument we make use of the fact that $f\left(
\theta;x,y\right)  =h_{\theta}\left(  0,\left(  \theta;x,y\right)  \right)  $,
which means that all properties of $h_{\theta}$ from Definition
\ref{def:covering-theta} hold for $f$.

The first possibility how $\mathbf{x}$ can lie on $\partial X$ is that
$\left\Vert x_{0}\right\Vert _{u}=1$. In this case, $\left(  \Theta_{0}%
;x_{0},0\right)  \in D_{\Theta_{0}}^{-}$, and by the first condition from
Definition \ref{def:covering-theta} we know that $f\left(  \Theta_{0}%
;x_{0},0\right)  \cap D=\emptyset$, so $\mathbf{f}\left(  \Theta_{0}%
;x_{0},0,0\right)  =\left(  f\left(  \Theta_{0};x_{0},0\right)  ,0\right)
\notin\mathbf{D}$, hence $F_{1}\left(  \mathbf{x}\right)  \neq0$, so in turn
$F\left(  \mathbf{x}\right)  \notin\overline{Z}$.

The second way that $\mathbf{x}$ can lie on $\partial X$ is that $\left\Vert
x_{i}\right\Vert _{u}=1$ for some $i\in\left\{  1,\ldots,k-1\right\}  $. Then
$\left(  \theta_{i};x_{i},y_{i}\right)  \in D_{\theta_{i}}^{-}$, and also from
condition one of Definition \ref{def:covering-theta} we have $f\left(
\theta_{i};x_{i},y_{i}\right)  \notin D$ so $\mathbf{f}\left(  \theta
_{i};x_{i},y_{i},z_{i}\right)  =\left(  f\left(  \theta_{i};x_{i}%
,y_{i}\right)  ,0\right)  \notin\mathbf{D}$. This means that $F_{i+1}\left(
\mathbf{x}\right)  \neq0$, hence $F\left(  \mathbf{x}\right)  \notin
\overline{Z}$.

If $\mathbf{x}$ lies in $\partial X$ because $\left\Vert x_{k}\right\Vert
_{u}=1$, then $\left(  \theta_{k};x_{k},y_{k}\right)  \in D_{\theta_{i}}^{-}$
and from Definition \ref{def:covering-theta} we see that $F_{k+1}\left(
\mathbf{x}\right)  =f\left(  \theta_{k};x_{k},y_{k}\right)  \notin D,$ hence
$F\left(  \mathbf{x}\right)  \notin\overline{Z}$.

Another possibility for $\mathbf{x}$ to be in $\partial X$ is to have
$\left\Vert y_{i}\right\Vert _{s}=1$ for some $i\in\left\{  1,\ldots
,k\right\}  $. From Definition \ref{def:covering-theta} it follows that
$f\left(  D\right)  \cap D^{+}=\emptyset$. We see that since $\left\Vert
y_{i}\right\Vert _{s}=1$ we have $F_{i}\left(  \mathbf{x}\right)  =\left(
f\left(  \theta_{i-1};x_{i-1},y_{i-1}\right)  ,0\right)  -\left(  \theta
_{i};x_{i},y_{i},z_{i}\right)  \neq0$ (where $y_{0} = 0$), so $F\left(
\mathbf{x}\right)  \notin\overline{Z}$.

The last possibility for $\mathbf{x}$ to be on $\partial X$ is that $\left\Vert z_{i}\right\Vert _{\mathbb{R}^{2n}%
}=\delta$ for some $i\in\left\{  1,\ldots,k\right\}  $, then $F_{i}\left(
\mathbf{x}\right)  =\left(  f\left(  \theta_{i-1};x_{i-1},y_{i-1}\right)
,0\right)  -\left(  \theta_{i};x_{i},y_{i},z_{i}\right)  \neq0$, so $F\left(
\mathbf{x}\right)  \notin\overline{Z}$.

Above we have shown that
\begin{equation}
F\left(  \partial X\right)  \cap\overline{Z}=\emptyset. \label{eq:F-admisible}%
\end{equation}
We now need to show that
\begin{equation}
F\left(  \overline{X}\right)  \cap\partial Z=\emptyset.
\label{eq:F-admisible-2}%
\end{equation}
If $\mathbf{y}\in\partial Z$, then
\begin{equation}
\mathbf{y}=\left(  0,\ldots,0,\left(  \theta;0,y\right)  \right)
\quad\text{for some }\theta\in\Lambda, \text{ and }\;\left\Vert y\right\Vert _{s}=1.
\label{eq:y-from-boundary}%
\end{equation}
Since $F_{k+1}\left(  \mathbf{x}\right)  :=f\left(  \theta_{k};x_{k}%
,y_{k}\right)  $ and from Definition \ref{def:covering-theta} it follows that
$f\left(  D\right)  \cap D^{+}=\emptyset$, we see that $F_{k+1}\left(
\mathbf{x}\right)  \neq\mathbf{y.}$ We have shown (\ref{eq:F-admisible-2}),
thus $F$ is admissible.

Our objective will now be to construct an admissible (in the sense of
Definition \ref{def:admissible-H}) homotopy from $F$ to some map that is
transversal to $Z$. We will do this in a number of steps, by constructing
several admissible homotopies and then gluing them together. A less patient
reader might want to take a peek at (\ref{eq:objective-of-homotopies}), where
we write out the map we make the homotopy to. Looking at
(\ref{eq:objective-of-homotopies}) will give an idea of our final objective.

Our first homotopy will be denoted as
\[
H^{\left(  1\right)  }=\left(  H_{1}^{\left(  1\right)  },\ldots
,H_{k}^{\left(  1\right)  },H_{k+1}^{\left(  1\right)  }\right)  :\left[
0,1\right]  \times\overline{X}\rightarrow Y.
\]
Since $D_{\Theta_{0}}\overset{f}{\Longrightarrow}D$, we can take the homotopy
$h_{\Theta_{0}}$ from Definition \ref{def:covering-theta}, and for
$\mathbf{x}$ of the form (\ref{eq:x-form}) we can define%
\begin{eqnarray*}
H_{1}^{\left(  1\right)  }\left(  \alpha,\mathbf{x}\right)   &  :=\left(
h_{\Theta_{0}}\left(  \alpha,\left(  \Theta_{0};x_{0},0\right)  \right)
,0\right)  -\left(  \theta_{1};x_{1},y_{1},z_{1}\right)  ,\\
H_{i}^{\left(  1\right)  }\left(  \alpha,\mathbf{x}\right)   &  :=F_{i}\left(
\mathbf{x}\right)  \qquad\text{for }i\neq1.
\end{eqnarray*}
Our homotopy is such that
\[
H^{\left(  1\right)  }\left(  0,\mathbf{x}\right)  =F\left(  \mathbf{x}%
\right)  ,
\]
and for some $\Theta_{1}\in\Lambda$ and linear $A_{0}:E_{\Theta_{0}}%
^{u}\rightarrow E_{\Theta_{1}}^{u}$ ($\Theta_{1}$ and $A_{0}$ follow from
Definition \ref{def:covering-theta})%
\[
H_{1}^{\left(  1\right)  }\left(  1,\mathbf{x}\right)  =\left(  \Theta
_{1};A_{0}x_{0},0,0\right)  -\left(  \theta_{1};x_{1},y_{1},z_{1}\right)  .
\]

We need to show that $H^{\left(  1\right)  }$ is admissible. This will follow
from an analogous argument to the one used to prove (\ref{eq:F-admisible}--\ref{eq:F-admisible-2}). We
first need to show that
\begin{equation}
H^{\left(  1\right)  }\left(  \alpha,\mathbf{x}\right)  \cap\overline
{Z}=\emptyset\quad\text{for }\mathbf{x}\in\partial X\text{ and }\alpha
\in\left[  0,1\right]  . \label{eq:H0-admissible}%
\end{equation}
We have already established (\ref{eq:F-admisible}) and we know that for
$i\neq1,$ by definition, $H_{i}^{\left(  1\right)  }\left(  \alpha
,\mathbf{x}\right)  =F_{i}\left(  \mathbf{x}\right)  $. This means that to
check (\ref{eq:H0-admissible}) it is enough to consider three cases. The first
is that $\mathbf{x}\in\partial X$ is such that $\left\Vert x_{0}\right\Vert
_{u}=1$. The second case $\left\Vert y_{1}\right\Vert _{s}=1$. The third is
$\left\Vert z_{1}\right\Vert _{\mathbb{R}^{2n}}=\delta$. (For all other
$\mathbf{x}\in\partial X$ condition (\ref{eq:H0-admissible}) follows from
(\ref{eq:F-admisible}).) In the first case $\left(  \Theta_{0};x_{0},0\right)
\in D_{\Theta_{0}}^{-}$ so since $h_{\Theta_{0}}\left(  \alpha,D_{\Theta_{0}%
}^{-}\right)  \cap D=\emptyset$ we obtain%
\begin{equation}
H_{1}^{\left(  1\right)  }\left(  \alpha,\mathbf{x}\right)  =\left(
h_{\Theta_{0}}\left(  \alpha,\left(  \Theta_{0};x_{0},0\right)  \right)
,0\right)  -\left(  \theta_{1};x_{1},y_{1},z_{1}\right)  \neq0,
\label{eq:H1-not-zero}%
\end{equation}
hence $H^{\left(  1\right)  }\left(  \alpha,\mathbf{x}\right)  \notin
\overline{Z}$. For the second case, since $h_{\Theta_{0}}\left(  \alpha
,D_{\Theta_{0}}\right)  \cap D^{+}=\emptyset$ also ensures
(\ref{eq:H1-not-zero}), we have $H^{\left(  1\right)  }\left(  \alpha
,\mathbf{x}\right)  \notin\overline{Z}$. If $\left\Vert z_{1}\right\Vert
=\delta$ then we also see that (\ref{eq:H1-not-zero}) holds. We have thus
established (\ref{eq:H0-admissible}). The fact that%
\begin{equation}
H^{\left(  1\right)  }\left(  \left[  0,1\right]  \times \overline{X}\right)
\cap\partial Z=\emptyset\label{eq:H0-admissible-2}%
\end{equation}
follows from (\ref{eq:F-admisible-2}). (This is because $H^{(1)}_{k+1}=F_{k+1}$, and $F_{k+1}$ was used to establish  (\ref{eq:F-admisible-2}).)  This means that we have established that
$H^{\left(  1\right)  }$ is admissible.

Since $H^{\left(  1\right)  }$ is admissible and $H^{\left(  1\right)
}\left(  0,\cdot\right)  =F$, from the homotopy property of the intersection
number we obtain%
\begin{equation}
I_{2}\left(  F,X,Z\right)  =I_{2}\left(  H^{\left(  1\right)  }\left(
0,\cdot\right)  ,X,Z\right)  =I_{2}\left(  H^{\left(  1\right)  }\left(
1,\cdot\right)  ,X,Z\right)  . \label{eq:first-I2-link}%
\end{equation}

Before specifying the next homotopy we shall make use of the excision
property. For this we take a closed set $U_{\Theta_{1}}\subset\Lambda$ such
that $\mathrm{int}U_{\Theta_{1}}\neq\emptyset$ and $\Theta_{1}\in
\mathrm{int}U_{\Theta_{1}}$. We can take $U_{\Theta_{1}}$ small enough so that
it is in the domain of some trivialization of $E$ and so that it is
contractible to the point $\Theta_{1}$. Let us denote such a continuous
contraction by $g_{\Theta_{1}}:\left[  0,1\right]  \times U_{\Theta_{1}%
}\rightarrow U_{\Theta_{1}}$ for which $g_{\Theta_{1}}\left(  0,\theta\right)
=\theta$ and $g_{\Theta_{1}}\left(  1,\theta\right)  =\Theta_{1}$. We now
define a set $\overline{X^{\left(  1\right)  }}\subset\overline{X}$ as
\[
\overline{X^{\left(  1\right)  }}=D_{\Theta_{0}}^{u}\times\mathbf{D}%
_{U_{\Theta_{1}}}\times\underset{k-1}{\underbrace{\mathbf{D}\times\ldots
\times\mathbf{D}}}.
\]
We see that%
\[
X^{\left(  1\right)  }=\mathrm{int}\overline{X^{\left(  1\right)  }%
}=\mathrm{int}D_{\Theta_{0}}^{u}\times\mathrm{int}\mathbf{D}_{U_{\Theta_{1}}%
}\times\underset{k-1}{\underbrace{\mathrm{int}\mathbf{D}\times\ldots
\times\mathrm{int}\mathbf{D}}}.
\]
We will use the excision property to restrict $H^{\left(  1\right)  }\left(
1,\cdot\right)  $ from $X$ to $X^{\left(  1\right)  }$. For this we first need
to show that
\begin{equation}
H^{\left(  1\right)  }\left(  1,X\right)  \cap Z=H^{\left(  1\right)  }\left(
1,X^{\left(  1\right)  }\right)  \cap Z. \label{eq:excision-ok}%
\end{equation}
If we take some $\mathbf{x\in}X\setminus X^{\left(  1\right)  }$ of the form
(\ref{eq:x-form}), then $\theta_{1}\notin\mathrm{int}U_{\Theta_{1}}$, so in
particular $\theta_{1}\neq\Theta_{1}$. This means that
\[
H_{1}^{\left(  1\right)  }\left(  1,\mathbf{x}\right)  =\left(  \Theta
_{1};A_{0}x_{0},0,0\right)  -\left(  \theta_{1};x_{1},y_{1},z_{1}\right)
\neq0,
\]
which implies (\ref{eq:excision-ok}). To use the excision property we also need to check that
\begin{equation}
H^{\left(  1\right)  }\left(  1,\partial X^{\left(  1\right)  }\right)
\cap\overline{Z}=\emptyset. \label{eq:excision-ok-2}%
\end{equation}
If $\mathbf{x\in}\partial X^{\left(  1\right)  }\cap\partial X$, then
(\ref{eq:excision-ok-2}) follows from (\ref{eq:H0-admissible}). If
$\mathbf{x\in}\partial X^{\left(  1\right)  }\setminus\partial X$, then
$\theta_{1}\in\partial U_{\Theta_{1}}$ and $\theta_{1}\neq\Theta_{1}$ so
$H_{1}^{\left(  1\right)  }\left(  1,\mathbf{x}\right)  \neq0$, which implies
(\ref{eq:excision-ok-2}). We can now apply the excision property.

From the excision property it follows that%
\begin{equation}
I_{2}\left(  H^{\left(  1\right)  }\left(  1,\cdot\right)  ,X,Z\right)
=I_{2}\left(  H^{\left(  1\right)  }\left(  1,\cdot\right)  ,X^{\left(
1\right)  },Z\right)  . \label{eq:first-excision}%
\end{equation}

We are ready to define our second homotopy. We consider
\[
G^{\left(  1\right)  }=\left(  G_{1}^{\left(  1\right)  },\ldots
,G_{k}^{\left(  1\right)  },G_{k+1}^{\left(  1\right)  }\right)  :\left[
0,1\right]  \times\overline{X^{\left(  1\right)  }}\rightarrow Y,
\]
defined as%
\begin{eqnarray}
G_{2}^{\left(  1\right)  }\left(  \alpha,\mathbf{x}\right)   &  :=\mathbf{f}%
\left(  g_{\Theta_{1}}\left(  \alpha,\theta_{1}\right)  ;x_{1},y_{1}%
,z_{1}\right)  -\left(  \theta_{2};x_{2},y_{2},z_{2}\right)  ,\nonumber\\
G_{i}^{\left(  1\right)  }\left(  \alpha,\mathbf{x}\right)   &  :=H_{i}%
^{\left(  1\right)  }\left(  1,\mathbf{x}\right)  \qquad\text{for }i\neq2.
\label{eq:G1-H0-same-coord}%
\end{eqnarray}
To show that $G^{\left(  1\right)  }$ is an admissible homotopy we first need
that $G^{\left(  1\right)  }\left(  \left[  0,1\right]  \times \partial X^{\left(
1\right)  }\right)  \cap\overline{Z}=\emptyset$. It is enough to show that for
$\mathbf{x}$ with $\left(  \theta_{1};x_{1},y_{1},z_{1}\right)  \in
\partial\mathbf{D}_{U_{\Theta_{1}}}$ we have $G^{\left(  1\right)  }\left(
\alpha,\mathbf{x}\right)  \notin Z$. (We do not need to consider other
$\mathbf{x\in\partial}X^{\left(  1\right)  }$ since we have
(\ref{eq:G1-H0-same-coord}) and (\ref{eq:H0-admissible}).) If $\left(
\theta_{1};x_{1},y_{1},z_{1}\right)  \in\partial\mathbf{D}_{U_{\Theta_{1}}}$
then we have three possibilities which we consider below.

The first possibility is that $\left\Vert x_{1}\right\Vert _{u}=1$, so
$\left(  g_{\Theta_{1}}\left(  \alpha,\theta_{1}\right)  ;x_{1},y_{1}\right)
\in D_{g_{\Theta_{1}}\left(  \alpha,\theta_{1}\right)  }^{-}.$ Then, since we
have $D_{g_{\Theta_{1}}\left(  \alpha,\theta_{1}\right)  }\overset{f}%
{\Longrightarrow}D$, we see that $f\left(  D_{g_{\Theta_{1}}\left(  \alpha,\theta
_{1}\right)  }^{-}\right)  \cap D=\emptyset$, therefore
\[
G_{2}^{\left(  1\right)  }\left(  \alpha,\mathbf{x}\right)  =\left(  f\left(
g_{\Theta_{1}}\left(  \alpha,\theta_{1}\right)  ;x_{1},y_{1}\right)
,0\right)  -\left(  \theta_{2};x_{2},y_{2},z_{2}\right)  \neq0,
\]
which implies that $G^{\left(  1\right)  }\left(  \alpha,\mathbf{x}\right)
\notin\overline{Z}$.

The second possibility is that $\left\Vert y_{1}\right\Vert _{s}=1$ or
$\left\Vert z_{1}\right\Vert _{\mathbb{R}^{2n}}=\delta$. Then $\mathbf{x\in
}\partial X$ and by (\ref{eq:H0-admissible}) we obtain that%
\[
G_{1}^{\left(  1\right)  }\left(  \alpha,\mathbf{x}\right)  =H_{1}^{\left(
1\right)  }\left(  1,\mathbf{x}\right)  \notin\overline{Z}.
\]

The third and last possibility is that $\theta_{1}\in\partial U_{\Theta_{1}}$,
but then $\theta_{1}\neq\Theta_{1}$, so
\[
G_{1}^{\left(  1\right)  }\left(  \alpha,\mathbf{x}\right)  =H_{1}^{\left(
0\right)  }\left(  1,\mathbf{x}\right)  =\left(  \Theta_{1};A_{0}%
x_{0},0,0\right)  -\left(  \theta_{1};x_{1},y_{1},z_{1}\right)  \neq0,
\]
hence $G^{\left(  1\right)  }\left(  \alpha,\mathbf{x}\right)  \notin
\overline{Z}$.

We also need to show that $G^{\left(  1\right)  }\left(  \left[  0,1\right]
\times \overline{X^{\left(  1\right)  }}\right)  \cap\partial Z=\emptyset$. This
follows from (\ref{eq:F-admisible-2}) since $G_{k+1}^{\left(  1\right)
}=F_{k+1}$. We have thus shown that $G^{\left(  1\right)  }$ is an admissible
homotopy, so from the homotopy property we obtain that%
\begin{eqnarray}
I_{2}\left(  H^{\left(  0\right)  }\left(  1,\cdot\right)  ,X^{\left(
1\right)  },Z\right)  &=I_{2}\left(  G^{\left(  1\right)  }\left(
0,\cdot\right)  ,X^{\left(  1\right)  },Z\right) \nonumber \\ &=I_{2}\left(  G^{\left(
1\right)  }\left(  1,\cdot\right)  ,X^{\left(  1\right)  },Z\right)  .
\label{eq:second-I2-link}%
\end{eqnarray}
Combining (\ref{eq:second-I2-link}) with (\ref{eq:first-I2-link}) and
(\ref{eq:first-excision}) gives%
\begin{equation}
I_{2}\left(  F,X,Z\right)  =I_{2}\left(  G^{\left(  1\right)  }\left(
1,\cdot\right)  ,X^{\left(  1\right)  },Z\right)  . \label{eq:third-I2-link}%
\end{equation}

Observe that
\[
G_{2}^{\left(  1\right)  }\left(  1,\mathbf{x}\right)  =\mathbf{f}\left(
\Theta_{1};x_{1},y_{1},z_{1}\right)  -\left(  \theta_{2};x_{2},y_{2}%
,z_{2}\right)  .
\]
What is important for us is that we have the fixed $\Theta_{1}$ on the right
hand side of the above expression. This means that we can use the homotopy
$h_{\Theta_{1}}$ from $D_{\Theta_{1}}\overset{f}{\Longrightarrow}D$ to define%
\[
H^{\left(  2\right)  }=\left(  H_{1}^{\left(  2\right)  },\ldots
,H_{k}^{\left(  2\right)  },H_{k+1}^{\left(  2\right)  }\right)  :\left[
0,1\right]  \times\overline{X^{\left(  1\right)  }}\rightarrow Y
\]
as%
\begin{eqnarray}
H_{2}^{\left(  2\right)  }\left(  \alpha,\mathbf{x}\right)   &  :=\left(
h_{\Theta_{1}}\left(  \alpha,\left(  \Theta_{1};x_{1},y_{1}\right)  \right)
,0\right)  -\left(  \theta_{2};x_{2},y_{2},z_{2}\right)  , \label{eq:H2-def} \\
H_{i}^{\left(  2\right)  }\left(  \alpha,\mathbf{x}\right)   &  :=G_{i}%
^{\left(  1\right)  }\left(  1,\mathbf{x}\right)  \qquad\text{for }i\neq2. \nonumber
\end{eqnarray}
Showing that $H^{\left(  2\right)  }$ is an admissible homotopy follows from
mirror steps to establishing that $H^{\left(  1\right)  }$ was admissible.
Thus%
\[
I_{2}\left(  G^{\left(  1\right)  }\left(  1,\cdot\right)  ,X^{\left(
1\right)  },Z\right)  =I_{2}\left(  H^{\left(  2\right)  }\left(
0,\cdot\right)  ,X^{\left(  1\right)  },Z\right)  =I_{2}\left(  H^{\left(
2\right)  }\left(  1,\cdot\right)  ,X^{\left(  1\right)  },Z\right)  ,
\]
hence by (\ref{eq:third-I2-link}) we have%
\begin{equation}
I_{2}\left(  F,X,Z\right)  =I_{2}\left(  H^{\left(  2\right)  }\left(
1,\cdot\right)  ,X^{\left(  1\right)  },Z\right)  . \label{eq:4th-I2-link}%
\end{equation}

Observe that
\begin{eqnarray*}
H_{2}^{\left(  2\right)  }\left(  0,\mathbf{x}\right)   &  =G^{\left(
1\right)  }\left(  1,\mathbf{x}\right)  ,\\
H_{2}^{\left(  2\right)  }\left(  1,\mathbf{x}\right)   &  =\left(  \Theta
_{2};A_{1}x_{1},0,0\right)  -\left(  \theta_{2};x_{2},y_{2},z_{2}\right)  ,
\end{eqnarray*}
where $\Theta_{2}$ and $A_{1}:E_{\Theta_{1}}^{u}\rightarrow E_{\Theta_{2}}%
^{u}$ result from the homotopy $h_{\Theta_{1}}$ from Definition \ref{def:covering-theta}. This means that we can take
an excision to
\[
\overline{X^{\left(  2\right)  }}:=D_{\Theta_{0}}^{u}\times\mathbf{D}%
_{U_{\Theta_{1}}}\times\mathbf{D}_{U_{\Theta_{2}}}\times\underset
{k-2}{\underbrace{\mathbf{D}\times\ldots\times\mathbf{D}}},
\]
where $U_{\Theta_{2}}\subset\Lambda$ is a closure of some small enough open
set around $\Theta_{2}$, which is contractible to the point $\Theta_{2}$ via a
homotopy $g_{\Theta_{2}}\left(  \alpha,\theta\right)  $. Using the same
arguments to those that lead to (\ref{eq:first-excision}) we obtain%
\[
I_{2}\left(  H^{\left(  2\right)  }\left(  1,\cdot\right)  ,X^{\left(
1\right)  },Z\right)  =I_{2}\left(  H^{\left(  2\right)  }\left(
1,\cdot\right)  ,X^{\left(  2\right)  },Z\right)  ,
\]
and by (\ref{eq:4th-I2-link})%
\[
I_{2}\left(  F,X,Z\right)  =I_{2}\left(  H^{\left(  2\right)  }\left(
1,\cdot\right)  ,X^{\left(  2\right)  },Z\right)  .
\]

We can now iterate the above construction step by step by taking, for
$j=2,\ldots k,$ the sets
\[
\overline{X^{\left(  j\right)  }}:=D_{\Theta_{0}}^{u}\times\underset
{j}{\underbrace{\mathbf{D}_{U_{\Theta_{1}}}\times\ldots\times\mathbf{D}%
_{U_{\Theta_{j}}}}}\times\underset{k-j}{\underbrace{\mathbf{D}\times
\ldots\times\mathbf{D}}},
\]
and admissible homotopies%
\begin{eqnarray*}
H^{\left(  j\right)  }  &  :\left[  0,1\right]  \times X^{\left(  j-1\right)
}\rightarrow Y,\\
G^{\left(  j\right)  }  &  :\left[  0,1\right]  \times X^{\left(  j\right)
}\rightarrow Y,
\end{eqnarray*}
defined as (compare with (\ref{eq:G1-H0-same-coord}) and (\ref{eq:H2-def}))
\begin{eqnarray*}
H_{j}^{\left(  j\right)  }\left(  \alpha,\mathbf{x}\right)   &  :=\left(
h_{\Theta_{j-1}}\left(  \alpha,\left(  \Theta_{j-1};x_{j-1},y_{j-1}\right)
\right)  ,0\right)  -\left(  \theta_{j};x_{j},y_{j},z_{j}\right)  ,\\
H_{i}^{\left(  j\right)  }\left(  \alpha,\mathbf{x}\right)   &  :=G_{i}%
^{\left(  j-1\right)  }\left(  \alpha,\mathbf{x}\right)  \qquad\text{for
}i\neq j,\\
& \\
G_{j+1}^{\left(  j\right)  }\left(  \alpha,\mathbf{x}\right)   &
:=\mathbf{f}\left(  g_{\Theta_{j}}\left(  \alpha,\theta_{j}\right)
;x_{j},y_{j},z_{j}\right)  -\left(  \theta_{j+1};x_{j+1},y_{j+1}%
,z_{j+1}\right)  ,\\
G_{i}^{\left(  j\right)  }\left(  \alpha,\mathbf{x}\right)   &  :=H_{i}%
^{\left(  j\right)  }\left(  1,\mathbf{x}\right)  \qquad\text{for }i\neq j+1.
\end{eqnarray*}

We sum up what we have achieved so far:%
\begin{eqnarray}
I_{2}\left(  F,X,Z\right)   &  =I_{2}\left(  H^{\left(  1\right)  }\left(
0,\cdot\right)  ,X,Z\right)  =I_{2}\left(  H^{\left(  1\right)  }\left(
1,\cdot\right)  ,X,Z\right) \nonumber\\
&  =I_{2}\left(  H^{\left(  1\right)  }\left(  1,\cdot\right)  ,X^{\left(
1\right)  },Z\right)  \qquad\text{(excision)}\nonumber\\
&  =I_{2}\left(  G^{\left(  1\right)  }\left(  0,\cdot\right)  ,X^{\left(
1\right)  },Z\right)  =I_{2}\left(  G^{\left(  1\right)  }\left(
1,\cdot\right)  ,X^{\left(  1\right)  },Z\right) \nonumber\\
&  =I_{2}\left(  H^{\left(  2\right)  }\left(  0,\cdot\right)  ,X^{\left(
1\right)  },Z\right)  =I_{2}\left(  H^{\left(  2\right)  }\left(
1,\cdot\right)  ,X^{\left(  1\right)  },Z\right) \nonumber\\
&  =I_{2}\left(  H^{\left(  2\right)  }\left(  1,\cdot\right)  ,X^{\left(
2\right)  },Z\right)  \qquad\text{(excision)\thinspace}%
\label{eq:homotopy-steps}\\
&  =I_{2}\left(  G^{\left(  2\right)  }\left(  0,\cdot\right)  ,X^{\left(
2\right)  },Z\right)  =I_{2}\left(  G^{\left(  2\right)  }\left(
1,\cdot\right)  ,X^{\left(  2\right)  },Z\right) \nonumber\\
&  =I_{2}\left(  H^{\left(  3\right)  }\left(  0,\cdot\right)  ,X^{\left(
2\right)  },Z\right)  =I_{2}\left(  H^{\left(  3\right)  }\left(
1,\cdot\right)  ,X^{\left(  2\right)  },Z\right) \nonumber\\
&  \;\;\vdots\nonumber\\
&  =I_{2}\left(  H^{\left(  k\right)  }\left(  1,\cdot\right)  ,X^{\left(
k\right)  },Z\right)  \qquad\text{(excision)}\nonumber\\
&  =I_{2}\left(  G^{\left(  k\right)  }\left(  0,\cdot\right)  ,X^{\left(
k\right)  },Z\right)  =I_{2}\left(  G^{\left(  k\right)  }\left(
1,\cdot\right)  ,X^{\left(  k\right)  },Z\right)  .\nonumber
\end{eqnarray}
We finally consider the last homotopy%
\[
H^{(k+1)}:\left[  0,1\right]  \times\overline{X^{\left(  k\right)  }%
}\rightarrow Y,
\]
defined as%
\begin{eqnarray*}
H_{k+1}^{\left(  k+1\right)  }\left(  \alpha,\mathbf{x}\right)   &
:=h_{\Theta_{k}}\left(  \alpha,\left(  \Theta_{k};x_{k},y_{k}\right)  \right)
,\\
H_{i}^{\left(  k+1\right)  }\left(  \alpha,\mathbf{x}\right)   &
:=G_{i}^{\left(  k\right)  }\left(  1,\mathbf{x}\right)  \qquad\text{for
}i\neq k+1.
\end{eqnarray*}
Showing that $H^{\left(  k+1\right)  }$ is admissible follows from analogous
argument to showing that $H^{(1)}$ is admissible. We therefore have%
\begin{eqnarray}
I_{2}\left(  G^{\left(  k\right)  }\left(  1,\cdot\right)  ,X^{\left(
k\right)  },Z\right)  &=I_{2}\left(  H^{\left(  k+1\right)  }\left(
0,\cdot\right)  ,X^{\left(  k\right)  },Z\right) \nonumber \\ &=I_{2}\left(  H^{\left(
k+1\right)  }\left(  1,\cdot\right)  ,X^{\left(  k\right)  },Z\right)  .
\label{eq:final-hom-link}%
\end{eqnarray}

What is important for us is that at the end of our construction we have
achieved:
\begin{eqnarray}
H_{1}^{\left(  k+1\right)  }\left(  1,\mathbf{x}\right)   &  =\left(
\Theta_{1};A_{0}x_{0},0,0\right)  -\left(  \theta_{1};x_{1},y_{1}%
,z_{1}\right)  ,\nonumber\\
H_{2}^{\left(  k+1\right)  }\left(  1,\mathbf{x}\right)   &  =\left(
\Theta_{2};A_{1}x_{1},0,0\right)  -\left(  \theta_{2};x_{2},y_{2}%
,z_{2}\right)  ,\nonumber\\
&  \;\;\vdots\label{eq:objective-of-homotopies}\\
H_{k}^{\left(  k+1\right)  }\left(  1,\mathbf{x}\right)   &  =\left(
\Theta_{k};A_{k-1}x_{k-1},0,0\right)  -\left(  \theta_{k};x_{k},y_{k}%
,z_{k}\right)  ,\nonumber\\
H_{k+1}^{\left(  k+1\right)  }\left(  1,\mathbf{x}\right)   &  =\left(
\Theta_{k+1};A_{k}x_{k},0\right)  .\nonumber
\end{eqnarray}
Since for $i=0,\ldots,k$, $A_{i}:E_{\Theta_{i}}^{u}\rightarrow E_{\Theta
_{i+1}}^{u}$ are linear and $A_{i}\left(  \partial B_{\Theta_{i}}^{u}\right)
\subset E_{\Theta_{i+1}}^{u}\setminus B_{\Theta_{i+1}}^{u}$, there is a unique
transversal intersection 
of
$H^{\left(  k+1\right)  }\left(  1,X^{(k)}\right)  $ with $Z$ at the point
$H^{\left(  k+1\right)  }\left(  1,\mathbf{x}^{\ast}\right)  $ for%
\[
\mathbf{x}^{\ast}=\left(  \left(  \Theta_{0};0\right)  ,\left(  \Theta
_{1};0,0,0\right)  ,\ldots,\left(  \Theta_{k};0,0,0\right)  \right)  \in
X^{\left(  k\right)  }.
\]
This means that $I_{2}\left(  H^{\left(  k+1\right)  }\left(  1,\cdot\right)
,X^{\left(  k\right)  },Z\right)  =1$, hence by (\ref{eq:homotopy-steps}%
--\ref{eq:final-hom-link})
\[
I_{2}\left(  F,X,Z\right)  =I_{2}\left(  H^{\left(  k+1\right)  }\left(
1,\cdot\right)  ,X^{\left(  k\right)  },Z\right)  =1.
\]
From the intersection property we therefore obtain an $\mathbf{x}\in X$ for
which we have (\ref{eq:Fx_in_Z}).

By establishing (\ref{eq:Fx_in_Z}) we have shown that for any $k\in\mathbb{N}$
there exists a trajectory starting from some $v_{k}\in D_{\Theta_{0}}$ for
which $f^{i}\left(  v_{k}\right)  \in D$ for $i=1,\ldots,k$. Since
$D_{\Theta_{0}}$ is compact, the claim of our theorem now simply follows by
passing to a limit $v^{*}\in D_{\theta}$ of a convergent subsequence of
$\{v_{k}\}_{k\in\mathbb{N}}$. For such a $v^{*}$, by continuity of $f$, we
will have $f^{i}\left(  v^{*}\right)  \in D$ for all $i\in\mathbb{N}$, as required.
\end{proof}

\section{Proof of Theorem \ref{th:invariance}}

The proof is similar to the one from the previous section. The difference is
that we will also \correction{comment 8}{keep track of }what is happening backwards in time while
setting up our maps and homotopies.

\label{proof:invariance}

\begin{proof}
Let us fix $\Theta_{0}\in\Lambda$. We start by showing that for a fixed
$k\in\mathbb{N}$ we have a sequence $\left\{  v_{i}\right\}  _{i=-k,\ldots
,k}\subset D$ such that $v_{0}\in D_{\Theta_{0}}$ and $f\left(  v_{i}\right)
=v_{i+1}$ for $i=-k,\ldots,k-1$. We define the sets%
\begin{eqnarray}
\overline{X}  &  =D^{u}\times\underset{k}{\underbrace{\mathbf{D}\times
\ldots\times\mathbf{D}}}\times\mathbf{D}_{\Theta_{0}}\times\underset
{k}{\underbrace{\mathbf{D}\times\ldots\times\mathbf{D}}} \label{eq:X-full-def}%
\\
Y  &  =\underset{2k+1}{\underbrace{\mathbb{R}^{2n}\times\ldots\times
\mathbb{R}^{2n}}}\times E\label{eq:Y-full-def}\\
Z  &  =\underset{2k+1}{\underbrace{\{0^{2n}\}\times\ldots\times\{0^{2n}\}}%
}\times\left\{  \left(  \theta;0,y\right)  |\;\theta\in\Lambda,\ \left\Vert
y\right\Vert _{s}<1\right\}  . \label{eq:Z-full-def}%
\end{eqnarray}
We see that%
\begin{eqnarray*}
\dim X  &  =\left(  c+u\right)  +2kn+\left(  u+s+n\right)  +2kn=(2k+1)2n+u,\\
\dim Z  &  =c+s,\\
\dim Y  &  =\left(  2k+1\right)  2n+n,
\end{eqnarray*}
therefore $X$ and $Z$ are manifolds of complementary dimensions with respect
to $Y$. $Y$ is a boundaryless manifold and $Z$ is its submanifold with
$\overline{Z}$ and $\partial Z$ of the form (\ref{eq:Z-1}--\ref{eq:Z-2}).

We define
\begin{equation}
F=\left(  F_{-k},\ldots,F_{k},F_{k+1}\right)  :\overline{X}\rightarrow Y
\label{eq:F-from-proof2}%
\end{equation}
as follows. For%
\begin{eqnarray*}
\mathbf{x}  &  =(\left(  \theta_{-k-1};x_{-k-1}\right)  ,\left(  \theta
_{-k};x_{-k},y_{-k},z_{-k}\right)  ,\ldots,\left(  \theta_{-1};x_{-1}%
,y_{-1},z_{-1}\right)  ,\\
&  \qquad\qquad\qquad\qquad\left(  \Theta_{0};x_{0},y_{0}%
,z_{0}\right)  ,\left(  \theta_{1};x_{1},y_{1},z_{1}\right)  ,\ldots,\left(
\theta_{k};x_{k},y_{k},z_{k}\right)  )
\end{eqnarray*}
we define\label{foot:th2}\footnote{To obtain the generalization stated in Remark \ref{rem:sequences} here we should use $\mathbf{f}_i$ in the definition of $F_i(\mathbf{x})$ for $i=-k,\ldots,k$ and $f_{k+1}$ in the definition of $F_{k+1}(\mathbf{x})$; throughout the reminder of the proof we would use homotopies resulting  from the coverings $D\overset{f_i}{\Longrightarrow}D$ in the respective places that follow.}
\begin{eqnarray*}
F_{-k}\left(  \mathbf{x}\right)   &  :=\mathbf{f}\left(  \theta_{-k-1}%
;x_{-k-1},0,0\right)  -\left(  \theta_{-k};x_{-k},y_{-k},z_{-k}\right)  ,\\
& \\
F_{-k+1}\left(  \mathbf{x}\right)   &  :=\mathbf{f}\left(  \theta_{-k}%
;x_{-k},y_{-k},z_{-k}\right)  -\left(  \theta_{-k+1};x_{-k+1},y_{-k+1}%
,z_{-k+1}\right)  ,\\
&  \;\;\;\vdots\\
F_{-1}\left(  \mathbf{x}\right)   &  :=\mathbf{f}\left(  \theta_{-2}%
;x_{-2},y_{-2},z_{-2}\right)  -\left(  \theta_{-1};x_{-1},y_{-1}%
,z_{-1}\right)  ,\\
& \\
F_{0}\left(  \mathbf{x}\right)   &  :=\mathbf{f}\left(  \theta_{-1}%
;x_{-1},y_{-1},z_{-1}\right)  -\left(  \Theta_{0};x_{0},y_{0},z_{0}\right)
,\\
F_{1}\left(  \mathbf{x}\right)   &  :=\mathbf{f}\left(  \Theta_{0};x_{0}%
,y_{0},z_{0}\right)  -\left(  \theta_{1};x_{1},y_{1},z_{1}\right)  ,\\
& \\
F_{2}\left(  \mathbf{x}\right)   &  :=\mathbf{f}\left(  \theta_{1};x_{1}%
,y_{1},z_{1}\right)  -\left(  \theta_{2};x_{2},y_{2},z_{2}\right)  ,\\
&  \;\;\;\vdots\\
F_{k}\left(  \mathbf{x}\right)   &  :=\mathbf{f}\left(  \theta_{k-1}%
;x_{k-1},y_{k-1},z_{k-1}\right)  -\left(  \theta_{k};x_{k},y_{k},z_{k}\right)
,\\
& \\
F_{k+1}\left(  \mathbf{x}\right)   &  :=f\left(  \theta_{k};x_{k}%
,y_{k}\right)  .
\end{eqnarray*}

If we find a point $\mathbf{x} \in\overline{X}$ for which $F\left(
\mathbf{x}\right)  \in Z,$ then, by Lemma \ref{lem:emb-the-same}, we will
obtain a finite trajectory (of length $2k+1$) of $f$ which remains in $D$. The
way in which we have chosen $F_{0}$ and $F_{1}$ has a special role. The
condition that $F_{0}=0$ ensures that the trajectory of $f$ reaches
$D_{\Theta_{0}}$. In $F_{1}$ we also find $\Theta_{0}$; this ensures that the
trajectory that reached $D_{\Theta_{0}}$ (because of $F_{0} =0$) will now
exits $D_{\Theta_{0}}$ in the next iterate.

Our objective is to show that $F\left(  X\right)  \cap Z\neq\emptyset$. We
will show this by proving that $I_{2}\left(  F,X,Z\right)  =1$. For this we
construct a sequence of admissible homotopies to a map for which it is easy to
compute the intersection number directly.

Our first homotopy $H^{\left(  0\right)  }:\left[  0,1\right]  \times
\overline{X}\rightarrow Y$ is defined as%
\begin{eqnarray*}
H_{-k}^{\left(  0\right)  }\left(  \alpha,\mathbf{x}\right)   &  :=\left(
h\left(  \alpha,\left(  \theta_{-k-1};x_{-k-1},0\right)  \right)  ,0\right)
-\left(  \theta_{-k};x_{-k},y_{-k},z_{-k}\right)  ,\\
H_{-k+1}^{\left(  0\right)  }\left(  \alpha,\mathbf{x}\right)   &  :=\left(
h\left(  \alpha,\left(  \theta_{-k};x_{-k},y_{-k}\right)  \right)  ,0\right)
-\left(  \theta_{-k+1};x_{-k+1},y_{-k+1},z_{-k+1}\right)  ,\\
&  \;\;\;\vdots\\
H_{-1}^{\left(  0\right)  }\left(  \alpha,\mathbf{x}\right)   &  :=\left(
h\left(  \alpha,\left(  \theta_{-2};x_{-2},y_{-2}\right)  \right)  ,0\right)
-\left(  \theta_{-1};x_{-1},y_{-1},z_{-1}\right)  ,\\
H_{0}^{\left(  0\right)  }\left(  \alpha,\mathbf{x}\right)   &  :=\left(
h\left(  \alpha,\left(  \theta_{-1};x_{-1},y_{-1}\right)  \right)  ,0\right)
-\left(  \Theta_{0};x_{0},y_{0},z_{0}\right)  ,\\
H_{i}^{\left(  0\right)  }\left(  \alpha,\mathbf{x}\right)   &  :=F_{i}\left(
\mathbf{x}\right)  \qquad\text{for }i>0.
\end{eqnarray*}
The fact that this homotopy is admissible follows from $D\overset{f}{\Longrightarrow
}D$ (by using analogous arguments to those used to show that the homotopies
considered in the proof of Theorem \ref{th:forward-invariance}).

We now take the sequence of admissible homotopies and excisions $H^{\left(
1\right)  }$, $G^{(1)}\ldots H^{\left(  k\right)  }$, $G^{(k)}$, $H^{\left(
k+1\right)  }$ defined as in the proof of Theorem \ref{th:forward-invariance},
leaving the coordinates $-k,\ldots,0$ without any changes. While making the
excisions, we make them to sets of the form%
\[
\overline{X^{\left(  j\right)  }}:=D^{u}\times\underset{k}{\underbrace
{\mathbf{D}\times\ldots\times\mathbf{D}}}\times\mathbf{D}_{\Theta_{0}}%
\times\underset{j}{\underbrace{\mathbf{D}_{U_{\Theta_{1}}}\times\ldots
\times\mathbf{D}_{U_{\Theta_{j}}}}}\times\underset{k-j}{\underbrace
{\mathbf{D}\times\ldots\times\mathbf{D}}},
\]
for $j=1,\ldots,k.$ We thus find an admissible homotopy of $F$ to%
\begin{eqnarray*}
H_{-k}^{\left(  k+1\right)  }\left(  1,\mathbf{x}\right)   &  =(\eta\left(
\theta_{-k-1}\right)  ;A_{\theta_{-k-1}}x_{-k-1},0,0)-\left(  \theta
_{-k};x_{-k},y_{-k},z_{-k}\right)  ,\\
H_{-k+1}^{\left(  k+1\right)  }\left(  1,\mathbf{x}\right)   &  =\left(
\eta\left(  \theta_{-k}\right)  ;A_{\theta_{-k}}x_{-k},0,0\right)  -\left(
\theta_{-k+1};x_{-k+1},y_{-k+1},z_{-k+1}\right)  ,\\
&  \;\;\vdots\\
H_{-1}^{\left(  k+1\right)  }\left(  1,\mathbf{x}\right)   &  =\left(
\eta\left(  \theta_{-2}\right)  ;A_{\theta_{-2}}x_{-2},0,0\right)  -\left(
\theta_{-1};x_{-1},y_{-1},z_{-1}\right)  ,\\
H_{0}^{\left(  k+1\right)  }\left(  1,\mathbf{x}\right)   &  =\left(
\eta\left(  \theta_{-1}\right)  ;A_{\eta_{-1}}x_{-1},0,0\right)  -\left(
\Theta_{0};x_{0},y_{0},z_{0}\right)  ,\\
H_{1}^{\left(  k+1\right)  }\left(  1,\mathbf{x}\right)   &  =\left(
\Theta_{1};A_{0}x_{0},0,0\right)  -\left(  \theta_{1};x_{1},y_{1}%
,z_{1}\right)  ,\\
H_{2}^{\left(  k+1\right)  }\left(  1,\mathbf{x}\right)   &  =\left(
\Theta_{2};A_{1}x_{1},0,0\right)  -\left(  \theta_{2};x_{2},y_{2}%
,z_{2}\right)  ,\\
&  \;\;\vdots\\
H_{k}^{\left(  k+1\right)  }\left(  1,\mathbf{x}\right)   &  =\left(
\Theta_{k};A_{k-1}x_{k-1},0,0\right)  -\left(  \theta_{k};x_{k},y_{k}%
,z_{k}\right)  ,\\
H_{k+1}^{\left(  k+1\right)  }\left(  1,\mathbf{x}\right)   &  =\left(
\Theta_{k+1};A_{k}x_{k},0\right)  .
\end{eqnarray*}

We need to show that $I_{2}\left(  H^{\left(  k+1\right)  }\left(
1,\cdot\right)  ,X^{\left(  k\right)  },Z\right)  =1$.

Since $\deg_{2}\left(  \eta\right)  \neq0$, we have a smooth map $\widetilde{\eta
}_{0}:\Lambda\rightarrow\Lambda$, homotopic to $\eta$, so that $\Theta_{0}$ is
a regular value of $\widetilde{\eta}_{0}$ for which the set $\widetilde{\eta}_{0}%
^{-1}\left(  \Theta_{0}\right)  $ has an odd number of points\footnote{As
highlighted in Remark \ref{rem:generalization}, we could use alternative
assumptions for this part of the argument. It would be enough if the degree
was not zero at each point in $\Lambda$, instead of assuming that the
(global) degree is not zero. Also, in the setting of oriantable manifolds we
could use the Brouwer degree for this part of the argument. Then, instead of the mod $2$ intersection number, we would use the oriented intersection number throughout the proof. \label{footnote}}. For the same reason we have a smooth $\widetilde{\eta}%
_{1}:\Lambda\rightarrow\Lambda$ homotopic to $\eta$, for which each point in
$\widetilde{\eta}_{-1}^{-1}\left(  \widetilde{\eta}_{0}^{-1}\left(  \Theta_{0}\right)
\right)  $ is regular and again the number of points in $\widetilde{\eta}%
_{-1}^{-1}\left(  \widetilde{\eta}_{0}^{-1}\left(  \Theta_{0}\right)  \right)  $
is odd. Proceding inductively we find smooth $\widetilde{\eta}_{-i}$ homotopic and
arbitrarily close to $\eta$, such that the points in $\widetilde{\eta}_{-i}%
^{-1}\circ\ldots\circ\widetilde{\eta}_{0}^{-1}\left(  \Theta_{0}\right)  $ are
regular for $\widetilde{\eta}_{-i}$, and that their number is odd; we find such
maps for $i=0,\ldots,k$. This means that $H^{\left(  k+1\right)  }\left(
1,\cdot\right)  $ is homotopic through an admissible map to $H:\overline
{X^{\left(  k\right)  }}\rightarrow Y$ defined as%
\begin{eqnarray*}
H_{-k}\left(  \mathbf{x}\right)   &  =(\widetilde{\eta}_{-k}\left(  \theta
_{-k-1}\right)  ;A_{\theta_{-k-1}}x_{-k-1},0,0)-\left(  \theta_{-k}%
;x_{-k},y_{-k},z_{-k}\right)  ,\\
H_{-k+1}\left(  \mathbf{x}\right)   &  =\left(  \widetilde{\eta}_{-k+1}\left(
\theta_{-k}\right)  ;A_{\theta_{-k}}x_{-k},0,0\right)  -\left(  \theta
_{-k+1};x_{-k+1},y_{-k+1},z_{-k+1}\right)  ,\\
&  \;\;\vdots\\
H_{-1}\left(  \mathbf{x}\right)   &  =\left(  \widetilde{\eta}_{1}\left(
\theta_{-2}\right)  ;A_{\theta_{-2}}x_{-2},0,0\right)  -\left(  \theta
_{-1};x_{-1},y_{-1},z_{-1}\right)  ,\\
H_{0}\left(  \mathbf{x}\right)   &  =\left(  \widetilde{\eta}_{0}\left(
\theta_{-1}\right)  ;A_{\eta_{-1}}x_{-1},0,0\right)  -\left(  \Theta_{0}%
;x_{0},y_{0},z_{0}\right)  ,\\
H_{i}\left(  \mathbf{x}\right)   &  =H_{i}^{\left(  k+1\right)  }\left(
1,\mathbf{x}\right)  ,\qquad\text{for }i>0.
\end{eqnarray*}
We see that $H$ intersects transversely with $Z$ at $H\left(  \mathbf{x}%
\right)  $ for points of the form%
\begin{eqnarray}
\mathbf{x} &  \mathbf{=(}\left(  \lambda_{-k-1};0\right)  ,\left(
\lambda_{-k};0,0,0\right)  ,\ldots,\left(  \lambda_{-1};0,0,0\right)
,\label{eq:inter-points-final}\\
&  \qquad\left(  \Theta_{0};0,0,0\right)  ,\left(  \Theta_{1};0,0,0\right)
,\ldots\left(  \Theta_{k};0,0,0\right)  ),\nonumber
\end{eqnarray}
where $\lambda_{-k-1}\in\widetilde{\eta}_{-k}^{-1}\circ\ldots\circ\widetilde{\eta}%
_{0}^{-1}\left(  \Theta_{0}\right)  $ and $\lambda_{-i}=\widetilde{\eta}_{-i}%
\circ\ldots\circ\widetilde{\eta}_{-k}\left(  \lambda_{-k-1}\right)  $ for
$i=1,\ldots,k.$ The number of the points of the form
(\ref{eq:inter-points-final}) is equal to $\#\widetilde{\eta}_{-k}^{-1}\circ
\ldots\circ\widetilde{\eta}_{0}^{-1}\left(  \Theta_{0}\right)  $, which is odd,
and so $I_{2}\left(  H,X^{\left(  k\right)  },Z,\right)  =1.$ Since
\[
I_{2}\left(  F,X,Z\right)  =I_{2}\left(  H^{\left(  k+1\right)  }\left(
1,\cdot\right)  ,X^{\left(  k\right)  },Z\right)  =I_{2}\left(  H,X^{\left(
k\right)  },Z\right)  ,
\]
this implies that $I_{2}\left(  F,X,Z\right)  =1.$

Since $I_{2}\left(  F,X,Z\right)  =1$, we have established the existence of a
trajectory $\left\{  v_{i}\right\}  _{i=-k}^{k}$ in $D$, for which $v_{0}\in
D_{\Theta_{0}}$. Because this holds for any $k\in\mathbb{N}$, we obtain a
sequence of such $v_{0}$'s lying in $D_{\Theta_{0}}$ which depend on $k$. Our
claim now follows by passing to a limit of a convergent subsequence, by the
virtue of compactness of $D_{\Theta_{0}}$, to obtain a point $v_{0}^{\ast}\in
D_{\Theta_{0}}$ for which the full trajectory is contained in $D$.
\end{proof}

\section{Proof of Theorem \ref{th:forward-continuation}%
\label{proof:forward-continuation}}

Before we proceed with the proof, we shall need two auxiliary results. The
first is a classical lemma:

\begin{lemma}
\cite[(9.3) p.12]{Whyburn} \label{lem:Whyburn}(Whyburn's lemma) Assume that
$K$ is a compact metric space and $K_{0},$ $K_{1}$ two closed disjoint subsets
of $K$. Then either

\begin{enumerate}
\item there exists a component (maximal closed connected subset) of $K$
meeting $K_{0}$ and $K_{1}$,

\item or there exist two disjoint compact sets $\widehat{K_{0}}$ and
$\widehat{K_{1}}$ such that $K=\widehat{K_{0}}\cup\widehat{K_{1}}$ and
$K_{i}\subset\widehat{K_{i}}$ for $i=1,2$.
\end{enumerate}
\end{lemma}

The second result is a generalization of the homotopy property of the
intersection number. Let $X,Y,Z$ be as in section
\ref{sec:intersection-number}. On $\left[  0,1\right]  \times\overline{X}$
consider the topology induced from $\mathbb{R\times}\overline{X}$. (This means
in particular that $\partial\left(  \left[  0,1\right]  \times\overline
{X}\right)  =\left[  0,1\right]  \times\partial\overline{X}$.) Let
$V\subset\left[  0,1\right]  \times\overline{X}$ be open and for $\alpha
\in\left[  0,1\right]  $ let $V_{\alpha}=\left\{  x|\,\left(  \alpha,x\right)
\in V\right\}  $.\begin{figure}[ptb]
\begin{center}
\includegraphics[height=3cm]{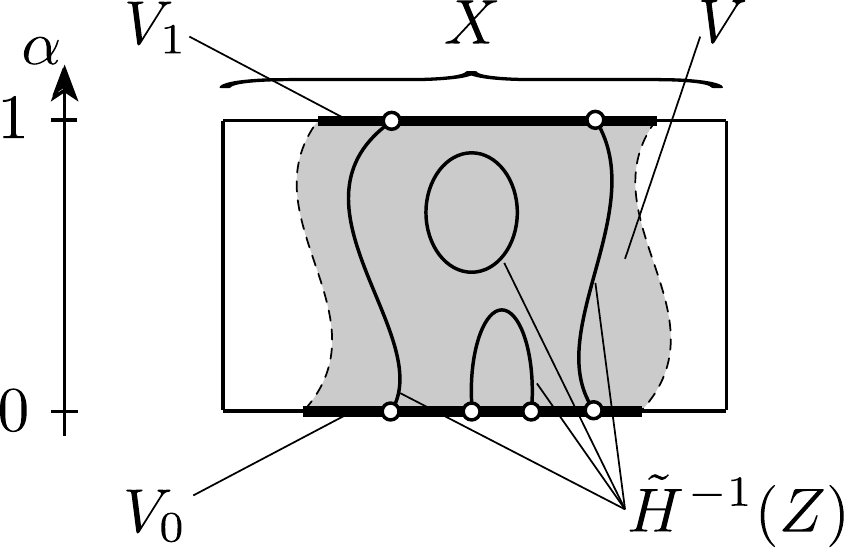}
\end{center}
\caption{Intuition behind the proof of Lemma \ref{lem:homotopy-generalised}.
The rectangle represents $[0,1]\times X$, the gray area is $V$ and the curves
contained in it are $\widetilde{H}^{-1}(Z)$. An important feature is that these
curves cannot pass through $\partial V$, which is represented by the dotted
lines. }%
\label{fig:generalised-homotopy}%
\end{figure}

\begin{lemma}
\label{lem:homotopy-generalised} If $H:\left[  0,1\right]  \times\overline
{X}\rightarrow Y$ is continuous, $H\left(  \partial V\right)  \cap\overline
{Z}=\emptyset$ and $H\left(  \overline{V}\right)  \cap\partial Z=\emptyset$
then
\[
I_{2}\left(  H\left(  0,\cdot\right)  ,V_{0},Z\right)  =I_{2}\left(  H\left(
1,\cdot\right)  ,V_{1},Z\right)  .
\]

\end{lemma}

\begin{proof}
\correction{comment 9}{The proof} follows from mirror arguments to the proof of the homotopy property
of the intersection number (see Lemma \ref{lem:inter-same} in  \ref{sec:app}). The intuition behind the proof is given in Figure
\ref{fig:generalised-homotopy}.

By performing an arbitrarily small modification of $H$ we can obtain
$\widetilde{H}$ for which $\widetilde{H}\left(  0,\cdot\right)  $ and $\widetilde
{H}\left(  1,\cdot\right)  $ are transversal to $Z$ and that $\widetilde{H}|_{V}$
is smooth and transversal to $Z$. We can make the modification small enough so
that for $\beta\in\left[  0,1\right]  $,
\begin{eqnarray*}
\left(  \left(  1-\beta\right)  H+\beta\widetilde{H}\right)  \left(  \partial
V\right)  \cap\overline{Z}  &  =\emptyset,\\
\left(  \left(  1-\beta\right)  H+\beta\widetilde{H}\right)  \left(  \overline
{V}\right)  \cap\partial Z  &  =\emptyset.
\end{eqnarray*}
This in particular implies that for $d=0,1$, $H\left(  d,\cdot\right)  $ and
$\widetilde{H}\left(  d,\cdot\right)  $ are homotopic through an admissible
homotopy, so%
\begin{equation}
I_{2}\left(  H\left(  d,\cdot\right)  ,V_{d},Z\right)  =I_{2}\left(  \widetilde
{H}\left(  d,\cdot\right)  ,V_{d},Z\right)  \qquad\text{for }d=0,1.
\label{eq:homot-tilde-cont}%
\end{equation}

Since $\widetilde{H}|_{V}$ is transversal to $Z$, we have that $\widetilde{H}%
^{-1}\left(  Z\right)  $ is a $1$-dimensional submanifold with boundary of
$V$, the boundary being (see Figure \ref{fig:generalised-homotopy})
\[
\partial\widetilde{H}^{-1}(Z)=\{0\}\times\widetilde{H}\left(  0,\cdot\right)
^{-1}(Z)\,\cup\,\{1\}\times\widetilde{H}\left(  1,\cdot\right)  ^{-1}(Z).
\]
By the classification of $1$-manifolds \cite{GP74}, $\partial\widetilde{H}%
^{-1}(Z)$ consists of an even number of points, hence
\[
\#\widetilde{H}\left(  0,\cdot\right)  ^{-1}(Z)\equiv\#\widetilde{H}\left(
1,\cdot\right)  ^{-1}(Z)\,\mathrm{mod}\,2.
\]
This by the intersection property for transversal maps means that%
\[
I_{2}\left(  \widetilde{H}\left(  0,\cdot\right)  ,V_{0},Z\right)  =I_{2}\left(
\widetilde{H}\left(  1,\cdot\right)  ,V_{1},Z\right)  ,
\]
which combined with (\ref{eq:homot-tilde-cont}) concludes our proof.
\end{proof}

The proof of Theorem \ref{th:forward-continuation} is based on the classical
ideas that stem from the Leray-Schauder continuation theorem \cite{Leray}.
This is a standard technique (see \cite{Mawhin} for an overview of related
results). We adopt it to be combined with the intersection number in our
particular setting.

\begin{proof}
[Proof of Theorem \ref{th:forward-continuation}]Let us fix $\theta=\Theta_{0}%
$. We will look for a connected component $C$ in the set $\left[  0,1\right]
\times E_{\Theta_{0}}$. In fact it will turn out that we can find $C$ in
$\left[  0,1\right]  \times\left(  E_{\Theta_{0}}^{u}\oplus\left\{  0\right\}
^{s}\right)  $. Let $D_{\Theta_{0}}^{u}:=\left\{  x\in E_{\Theta_{0}}%
^{u}|\,\left\Vert x\right\Vert _{u}\leq1\right\}  $. The set $D_{\Theta_{0}%
}^{u}$ is a compact metric space, with the metric defined by the norm on the
bundle $E_{\Theta_{0}}^{u}$. Let us equip $\left[  0,1\right]  \times
D_{\Theta_{0}}^{u}$ with a metric
\begin{equation}
m\left(  \left(  \alpha_{1},x_{1}\right)  ,\left(  \alpha_{2},x_{2}\right)
\right)  :=\max\{\left\vert \alpha_{1}-\alpha_{2}\right\vert ,\left\Vert
x_{1}-x_{2}\right\Vert _{u}\}, \label{eq:m-metric}%
\end{equation}
and define a set
\[
K:=\left\{  \left(  \alpha,x\right)  |\; x\in D_{\Theta_{0}}^{u},\ f_{\alpha
}^{i}\left(  \Theta_{0};x,0\right)  \in D\text{ for all }i\in\mathbb{N} \text{
and } \alpha\in[0,1] \right\}  .
\]
From the covering $D_{\Theta_{0}}\overset{f_{\alpha}}{\Longrightarrow}D$ it follows
that any point from $D_{\Theta_{0}}^{-}$ exits $D$, which implies%
\begin{equation}
K\cap\left(  \left[  0,1\right]  \times\partial D_{\Theta_{0}}^{u}\right)
=\emptyset. \label{eq:K-not-on-boundary}%
\end{equation}

Since the family $f_{\alpha}$ is continuous and $D$ is closed, if we take a
convergent sequence $(\alpha_{j},x_{j})\in K,$ then $\lim_{j\rightarrow\infty
}f_{\alpha_{j}}^{i}\left(  \Theta_{0};x_{j},0\right)  \in D$, so $K$ is a
compact metric space with the metric (\ref{eq:m-metric}). Let $L_{\alpha
},L_{\alpha}^{k}\subset D_{\Theta_{0}}^{u}$ be compact sets defined as%
\begin{eqnarray*}
L_{\alpha}:=\left\{  x\in D_{\Theta_{0}}^{u}|\;f_{\alpha}^{i}\left(
\Theta_{0};x,0\right)  \in D\text{ for all }i\in\mathbb{N}\right\}, \\
L_{\alpha}^{k}:=\left\{  x\in D_{\Theta_{0}}^{u}|\;f_{\alpha}^{i}\left(
\Theta_{0};x,0\right)  \in D\text{ for }i=0,\ldots,k\right\},
\end{eqnarray*}
for $\alpha\in\left[  0,1\right] $ and $k\in\mathbb{N}$.
Note that $L_{\alpha}\subset L_{\alpha}^{k}$ and $L_{\alpha}^{k+1}\subset
L_{\alpha}^{k}$ for $\alpha\in\left[  0,1\right]  $ and $k\in\mathbb{N}.$

Let $K_{0}:=\left\{  0\right\}  \times L_{0}$ and $K_{1}:=\left\{  1\right\}
\times L_{1}$. By Theorem \ref{th:forward-invariance}, $K_{0}$ and $K_{1}$ are
nonempty. (Here we in fact used the fact that in the proof we have established
that we can take the $v$ from the statement of Theorem
\ref{th:forward-invariance} to be from $E_{\Theta_{0}}^{u}\oplus\left\{
0\right\}  ^{s}\cap D$.) By Lemma \ref{lem:Whyburn} we have two possibilities.
The first ensures our claim, so we need to rule out the second one, which will
conclude our proof.

Suppose that we have two disjoint compact sets $\widehat{K_{0}}$ and
$\widehat{K_{1}}$ such that $K=\widehat{K_{0}}\cup\widehat{K_{1}}$ and
$K_{i}\subset\widehat{K_{i}}$ for $i=0,1$. Let us take small $\varepsilon$ so
that%
\begin{equation}
\varepsilon<\frac{1}{2}\mathrm{dist}\left(  \widehat{K_{0}},\widehat{K_{1}%
}\right)  .\label{eq:epsilon-Ki-sep}%
\end{equation}
Because of (\ref{eq:K-not-on-boundary}), we can take $\varepsilon>0$ small
enough so that in addition to (\ref{eq:epsilon-Ki-sep}) we have%
\[
U:=\left\{  \left(  \alpha,x\right)  \in\left[  0,1\right]  \times
D_{\Theta_{0}}^{u}|\;\mathrm{dist}\left(  \left(  \alpha,x\right)
,\widehat{K_{0}}\right)  <\varepsilon\right\}  \subset\left[  0,1\right]
\times\mathrm{int}D_{\Theta_{0}}^{u}.
\]
Clearly $K_{0}\subset U$ and also by (\ref{eq:epsilon-Ki-sep}) we see that
$K_{1}\cap U=\emptyset$. We shall use the notation $U_{\alpha}=\left\{
x|\,\left(  \alpha,x\right)  \in U\right\}  $, so we can rewrite the previous
statement as $L_{0}\subset U_{0}$ and $L_{1}\cap U_{1}=\emptyset$. Since%
\[
L_{\alpha}=\bigcap_{k=0}^{\infty}L_{\alpha}^{k}\qquad\text{for }\alpha
\in\left[  0,1\right]  ,
\]
by taking sufficiently large $k$ we will have
\begin{eqnarray}
L_{0}^{k} &  \subset U_{0},\label{eq:L0-in-U0}\\
L_{1}^{k}\cap U_{1} &  =\emptyset,\label{eq:L1-no-U1}%
\end{eqnarray}
and since $\partial U\cap K=\emptyset,$ we can also choose $k$ large enough so
that%
\begin{equation}
L_{\alpha}^{k}\cap\partial U_{\alpha}=\emptyset\qquad\text{for }\alpha
\in\left[  0,1\right]  .\label{eq:Lk_alpha-ok}%
\end{equation}

Consider $Y$ and $Z$ defined in (\ref{eq:Y-forward-def}%
--\ref{eq:Z-forward-def}), and take%
\begin{equation}
\overline{V}:=\overline{U}\times\underset{k}{\underbrace{\mathbf{D}%
\times\ldots\times\mathbf{D}}}. \label{eq:V-def-cont}%
\end{equation}
We shall consider $H=\left(  H_{1},\ldots,H_{k},H_{k+1}\right)  :\overline
{V}\rightarrow Y$ which is defined for points%
\[
\mathbf{x}=\left(  \left(  \Theta_{0};x_{0}\right)  ,\left(  \theta_{1}%
;x_{1},y_{1},z_{1}\right)  ,\ldots,\left(  \theta_{k};x_{k},y_{k}%
,z_{k}\right)  \right)  \in\overline{V}_{\alpha}%
\]
as
\begin{eqnarray*}
H_{1}\left(  \alpha,\mathbf{x}\right)   &  :=\left(  f_{\alpha}\left(
\Theta_{0};x_{0},0\right)  ,0\right)  -\left(  \theta_{1};x_{1},y_{1}%
,z_{1}\right)  ,\\
H_{2}\left(  \alpha,\mathbf{x}\right)   &  :=\left(  f_{\alpha}\left(
\theta_{1};x_{1},y_{1}\right)  ,0\right)  -\left(  \theta_{2};x_{2}%
,y_{2},z_{2}\right)  ,\\
&  \;\;\;\vdots\\
H_{k}\left(  \alpha,\mathbf{x}\right)   &  :=\left(  f_{\alpha}\left(
\theta_{k-1};x_{k-1},y_{k-1}\right)  ,0\right)  -\left(  \theta_{k}%
;x_{k},y_{k},z_{k}\right)  ,\\
H_{k+1}\left(  \alpha,\mathbf{x}\right)   &  :=f_{\alpha}\left(  \theta
_{k};x_{k},y_{k}\right)  .
\end{eqnarray*}

We will now show that
\begin{equation}
H\left(  \partial V\right)  \cap\overline{Z}=\emptyset.
\label{eq:H-admissible-cont}%
\end{equation}
This follows from mirror arguments to those used to show (\ref{eq:F-admisible}%
). The only difference is that we also need to consider the case when $\left(
\alpha,\mathbf{x}\right)  \in V$ is such that $\left(  \alpha,x_{0}\right)
\in\partial U$. In such case, due to (\ref{eq:Lk_alpha-ok}), we see that we
can not have $f_{\alpha}^{k}\left(  \Theta_{0};x_{0},0\right)  \in D$ so any
point for which $H\left(  \alpha,\mathbf{x}\right)  \in Z$ can not have
$\left(  \alpha,x_{0}\right)  \in\partial U$. We have thus shown
(\ref{eq:H-admissible-cont}).

From arguments identical to showing (\ref{eq:F-admisible-2}) we also obtain
\begin{equation}
H\left(  \overline{V}\right)  \cap\partial Z=\emptyset.
\label{eq:H-admissible-cont-2}%
\end{equation}

By Lemma (\ref{lem:homotopy-generalised}) together with
(\ref{eq:H-admissible-cont}--\ref{eq:H-admissible-cont-2}) we obtain%
\begin{equation}
I_{2}\left(  H\left(  0,\cdot\right)  ,V_{0},Z\right)  =I_{2}\left(  H\left(
1,\cdot\right)  ,V_{1},Z\right)  . \label{eq:contradiction1}%
\end{equation}

We will now compute $I_{2}\left(  H\left(  0,\cdot\right)  ,V_{0},Z\right)  $.
Let $\overline{X}$ be the set defined in (\ref{eq:X-forward-def}). The first
coordinate of $V_{0}$ is $U_{0}$ (see (\ref{eq:V-def-cont})). By
(\ref{eq:L0-in-U0}) $U_{0}$ contains all points $x_{0}$ such that $f_{0}%
^{i}\left(  \Theta_{0};x_{0},0\right)  \in D$ for every $i\in\left\{
1,\ldots,k\right\}  .$ This means that if $\mathbf{x}\in X\setminus V_{0}$
then we can not have $f_{0}^{i}\left(  \Theta_{0};x_{0},0\right)  \in D$ for
any $i\in\left\{  1,\ldots,k\right\}  $. Thus, for $\mathbf{x}\in X\setminus
V_{0}$ we can not have $H\left(  0,\mathbf{x}\right)  \in Z$, hence%
\[
H\left(  0,X\right)  \cap Z=H\left(  0,V_{0}\right)  \cap Z.
\]
Note that from (\ref{eq:H-admissible-cont}) $H\left(  0,\partial V_{0}\right)
\cap\overline{Z}=\emptyset$, so from the excision property%
\begin{equation}
I_{2}\left(  H\left(  0,\cdot\right)  ,X,Z\right)  =I_{2}\left(  H\left(
0,\cdot\right)  ,V_{0},Z\right)  . \label{eq:excision-to-V0}%
\end{equation}
In the proof of Theorem \ref{th:forward-invariance} we have established that
$I_{2}\left(  H\left(  0,\cdot\right)  ,X,Z\right)  =1,$ so by
(\ref{eq:excision-to-V0})%
\begin{equation}
I_{2}\left(  H\left(  0,\cdot\right)  ,V_{0},Z\right)  =1.
\label{eq:contradiction2}%
\end{equation}

We will now compute $I_{2}\left(  H\left(  1,\cdot\right)  ,V_{1},Z\right)  $.
The set $L_{1}^{k}$ contains all points $x_{0}$ for which $f_{1}^{i}\left(
\Theta_{0};x_{0},0\right)  \in D$ for all $i\in\left\{  1,\ldots,k\right\}  $.
If $\mathbf{x}\in V_{1}$, then $x_{0}\in U_{1}$, so by (\ref{eq:L1-no-U1}) we
can not have $f_{1}^{i}\left(  \Theta_{0};x_{0},0\right)  \in D$ for all
$i\in\left\{  1,\ldots,k\right\}  $. This means that for $\mathbf{x}\in V_{1}%
$, $H\left(  1,\mathbf{x}\right)  \notin Z$, so by the intersection property%
\begin{equation}
I_{2}\left(  H\left(  1,\cdot\right)  ,V_{1},Z\right)  =0.
\label{eq:contradiction3}%
\end{equation}

By (\ref{eq:contradiction1}), (\ref{eq:contradiction2}),
(\ref{eq:contradiction3}) we have obtained a contradiction. This means that we
have ruled out the second case of Lemma \ref{lem:Whyburn} and finished our proof.
\end{proof}


\section{Proof of Theorem \ref{th:continuation}\label{proof:continuation}}

\begin{proof}
The proof follows along the same lines as the proof of Theorem
\ref{th:forward-continuation}.

Let us fix $\theta=\Theta_{0}$. We will look for a connected component $C$
in $\left[  0,1\right]  \times E_{\Theta_{0}}$. The set $D_{\Theta_{0}}$ is a
compact metric space, with the metric defined by the norm on the bundle
$E_{\Theta_{0}}$. We equip $\left[  0,1\right]  \times D_{\Theta_{0}}$ with a
metric
\begin{equation}
m\left(  \left(  \alpha_{1},v_{1}\right)  ,\left(  \alpha_{2},v_{2}\right)
\right)  =\max\{\left\vert \alpha_{1}-\alpha_{2}\right\vert ,\left\Vert
v_{1}-v_{2}\right\Vert \}, \label{eq:m-metric-again}%
\end{equation}
and define a set
\begin{eqnarray*}
K &:=\{  \left(  \alpha,v\right)  |\;v\in D_{\Theta_{0}}\text{ and there exists
a trajectory of }f_{\alpha}\text{ in }D  \\ &\qquad \text{ passing through }v\}.
\end{eqnarray*}

We shall say that a sequence $\left\{  v_{i}\right\}  $ is a trajectory of
$f_{\alpha}$ of length $k$ in $D$ passing through $v$ if $v_{0}=v$, $v_{i}\in
D$ for $i=-k,\ldots,k$ and $f_{\alpha}\left(  v_{i}\right)  =v_{i+1}$ for
$i=-k,\ldots,k-1$.

The set $K$ is a compact metric space with the metric (\ref{eq:m-metric-again}%
). For $\alpha\in\left[  0,1\right]  $ and $k\in\mathbb{N}$ let
$L_{\alpha},L_{\alpha}^{k}\subset D_{\Theta_{0}}^{u}$ be compact sets
defined as%
\begin{eqnarray*}
L_{\alpha}  &  :=\left\{  v\in D_{\Theta_{0}}|\;\text{there exists a trajectory
of }f_{\alpha}\text{ in }D\text{ passing through }v\right\}  ,\\
L_{\alpha}^{k}  &  :=\{  v\in D_{\Theta_{0}}|\;\text{there exists a
trajectory of }f_{\alpha}\text{ of length }k\text{ in }D \\ & \qquad \text{ passing through
}v\}  .
\end{eqnarray*}
Note that $L_{\alpha}^{k+1}\subset L_{\alpha}^{k}$ and $L_{\alpha}\subset
L_{\alpha}^{k}$ for $\alpha\in\left[  0,1\right]  $ and $k\in\mathbb{N}.$

Let $K_{0}:=\left\{  0\right\}  \times L_{0}$ and $K_{1}:=\left\{  1\right\}
\times L_{1}$. By Theorem \ref{th:invariance}, $K_{0}$ and $K_{1}$ are
nonempty. By Lemma \ref{lem:Whyburn} we have two possibilities. The first
ensures our claim, so we need to rule out the second one, which will conclude
our proof.

Suppose that we have disjoint compact $\widehat{K_{0}},$ $\widehat{K_{1}}$
such that $K=\widehat{K_{0}}\cup\widehat{K_{1}}$ and $K_{0}\subset
\widehat{K_{0}}$, $K_{1}\subset\widehat{K_{1}}$. Consider $\varepsilon
<\frac{1}{2}\mathrm{dist}\left(  \widehat{K_{0}},\widehat{K_{1}}\right)  $, chosen sufficiently small so that%
\[
U:=\left\{  \left(  \alpha,v\right)  \in\left[  0,1\right]  \times
D_{\Theta_{0}} | \;\mathrm{dist}\left(  \left(  \alpha,v\right)
,\widehat{K_{0}}\right)  <\varepsilon  \right\}  \subset\left[
0,1\right]  \times\mathrm{int}D_{\Theta_{0}}.
\]
We shall embed $U$ in $\mathbb{R}^{2n}$ (we use Notation \ref{notation2})
\[
\mathbf{U:}=\left\{  \left(  \alpha,\left(  v,z\right)  \right)  \in\left[
0,1\right]  \times\mathcal{T}\,|\;\left(  \alpha,v\right)  \in U,\ \left\Vert
z\right\Vert _{\mathbb{R}^{2n}}\leq\delta\right\}  \subset\mathbb{R}^{2n}.
\]

Consider $Y$ and $Z$ defined in (\ref{eq:Y-full-def}--\ref{eq:Z-full-def}),
and take%
\[
\overline{V}:=D^{u}\times\underset{k}{\underbrace{\mathbf{D}\times\ldots
\times\mathbf{D}}}\times\overline{\mathbf{U}}\times\underset{k}{\underbrace
{\mathbf{D}\times\ldots\times\mathbf{D}}}.
\]
We shall consider $H=\left(  H_{-k},\ldots,H_{k},H_{k+1}\right)  :\overline
{V}\rightarrow Y$ which is defined for points%
\begin{eqnarray*}
\mathbf{x}  &  =(\left(  \theta_{-k-1};x_{-k-1}\right)  ,\left(  \theta
_{-k};x_{-k},y_{-k},z_{-k}\right)  ,\ldots,\left(  \theta_{-1};x_{-1}%
,y_{-1},z_{-1}\right)  ,\\
&  \qquad\qquad\qquad\left(  \Theta_{0};x_{0},y_{0}%
,z_{0}\right)  ,\left(  \theta_{1};x_{1},y_{1},z_{1}\right)  ,\ldots,\left(
\theta_{k};x_{k},y_{k},z_{k}\right)  )
\end{eqnarray*}
as (compare with (\ref{eq:F-from-proof2}) used in the proof of Theorem
\ref{th:invariance})%
\begin{eqnarray*}
H_{-k}\left(  \alpha,\mathbf{x}\right)   &  :=\left(  f_{\alpha}\left(
\theta_{-k-1};x_{-k-1},0\right)  ,0\right)  -\left(  \theta_{-k};x_{-k}%
,y_{-k},z_{-k}\right)  ,\\
H_{-k+1}\left(  \alpha,\mathbf{x}\right)   &  :=\left(  f_{\alpha}\left(
\theta_{-k};x_{-k},y_{-k}\right)  ,0\right)  -\left(  \theta_{-k+1}%
;x_{-k+1},y_{-k+1},z_{-k+1}\right)  ,\\
&  \;\;\;\vdots\\
H_{-1}\left(  \alpha,\mathbf{x}\right)   &  :=\left(  f_{\alpha}\left(
\theta_{-2};x_{-2},y_{-2}\right)  ,0\right)  -\left(  \theta_{-1}%
;x_{-1},y_{-1},z_{-1}\right)  ,\\
H_{0}\left(  \alpha,\mathbf{x}\right)   &  :=\left(  f_{\alpha}\left(
\theta_{-1};x_{-1},y_{-1}\right)  ,0\right)  -\left(  \Theta_{0};x_{0}%
,y_{0},z_{0}\right)  ,\\
H_{1}\left(  \alpha,\mathbf{x}\right)   &  :=\left(  f_{\alpha}\left(
\Theta_{0};x_{0},y_{0}\right)  ,0\right)  -\left(  \theta_{1};x_{1}%
,y_{1},z_{1}\right)  ,\\
H_{2}\left(  \alpha,\mathbf{x}\right)   &  :=\left(  f_{\alpha}\left(
\theta_{1};x_{1},y_{1}\right)  ,0\right)  -\left(  \theta_{2};x_{2}%
,y_{2},z_{2}\right)  ,\\
&  \;\;\;\vdots\\
H_{k}\left(  \alpha,\mathbf{x}\right)   &  :=\left(  f_{\alpha}\left(
\theta_{k-1};x_{k-1},y_{k-1}\right)  ,0\right)  -\left(  \theta_{k}%
;x_{k},y_{k},z_{k}\right)  ,\\
H_{k+1}\left(  \alpha,\mathbf{x}\right)   &  :=f_{\alpha}\left(  \theta
_{k};x_{k},y_{k}\right)  .
\end{eqnarray*}

From now on we skip the details since they follow along the same lines as in
the proof of Theorem \ref{th:forward-continuation}. We just outline the steps:
Using Lemma \ref{lem:homotopy-generalised} we can show that%
\begin{equation}
I_{2}\left(  H\left(  0,\cdot\right)  ,V_{0},Z\right)  =I_{2}\left(  H\left(
1,\cdot\right)  ,V_{1},Z\right)  . \label{eq:contr-cont-full-1}%
\end{equation}
Using the excision property, for $X$ defined in (\ref{eq:X-full-def}), we
obtain%
\begin{equation}
I_{2}\left(  H\left(  0,\cdot\right)  ,V_{0},Z\right)  =I_{2}\left(  H\left(
0,\cdot\right)  ,X,Z\right)  =1. \label{eq:contr-cont-full-2}%
\end{equation}
From the fact that $V_{1}$ can not contain trajectories of length $k$ in $D$
passing through $D_{\Theta_{0}}$ we also obtain%
\begin{equation}
I_{2}\left(  H\left(  1,\cdot\right)  ,V_{1},Z\right)  =0.
\label{eq:contr-cont-full-3}%
\end{equation}
Conditions (\ref{eq:contr-cont-full-1}--\ref{eq:contr-cont-full-3}) lead to a
contradiction, which concludes our proof.
\end{proof}


\section{Proof of Lemma \ref{lem:nhim-covers}.\label{sec:nhim-covers}}

\begin{proof}
Let $E=E^{u}\oplus E^{s}$ and consider $l:E\rightarrow E$ defined as
\[
l\left(  \theta;x,y\right)  :=\left(  f\left(  \theta\right)  ;df(\theta
)x,df(\theta)y\right)  .
\]
Note $l$ is well defined since the splitting (\ref{eq:splitting}) {is
}invariant under the action of the differential $df$. We shall refer to $l$ as
the `linearized map'. Note that
\[
l^{k}\left(  \theta;x,y\right)  =\left(  f^{k}\left(  \theta\right)
;df^{k}(\theta)x,df^{k}(\theta)y\right)  .
\]

For $r>0$ we define $D(r)\subset E$ as%
\[
D(r):=\left\{  \left(  \theta;x,y\right)  |\;\theta\in\Lambda,\ x\in
E_{\theta}^{u},\ y\in E_{\theta}^{s},\ \left\Vert x\right\Vert \leq
r,\ \left\Vert y\right\Vert \leq r\right\}  .
\]
We will show that for any $r>0$ we have
\begin{equation}
D(r)\overset{l^k}{\Longrightarrow}D(r).\label{eq:l-covering}%
\end{equation}
The homotopy $h$ (see Definition \ref{def:covering}) for the covering
(\ref{eq:l-covering}) can be taken as
\begin{equation}
h\left(  \beta,\left(  \theta;x,y\right)  \right)  :=\left(  f^{k}\left(
\theta\right)  ;df^{k}(\theta)x,(1-\beta)df^{k}(\theta)y\right)  ,\quad
\beta\in\left[  0,1\right]  .\label{eq:homotopy-for-nhim}%
\end{equation}

Before showing that $h$ satisfies all required conditions we note that since
$\lambda<1$, from $k>\log_{\lambda}\frac{1}{C}$ it follows that $C\lambda
^{k}<1$. Using (\ref{eq:expansion-condition}) we also see that for any
$\theta\in\Lambda$ and $x\in E_{\theta}^{u}$
\begin{equation}
\left\Vert x\right\Vert =\left\Vert d\left(  f^{-k}\circ f^{k}\right)
(\theta)x\right\Vert =\left\Vert df^{-k}(f^{k}(\theta))df^{k}(\theta
)x\right\Vert \leq C\lambda^{k}\left\Vert df^{k}(\theta)x\right\Vert
.\label{eq:expansion-condition-2}%
\end{equation}

We will now show that $h$ satisfies conditions from Definition
\ref{def:covering}. If $\left(  \theta;x,y\right)  \in D^{-}\left(  r\right)
$, meaning that $\left\Vert x\right\Vert =r,$ then by
(\ref{eq:expansion-condition-2}), $\left\Vert df^{k}(\theta)x\right\Vert
>\left(  C\lambda^{k}\right)  ^{-1}\left\Vert x\right\Vert >r$, hence
$h\left(  \beta,\left(  \theta;x,y\right)  \right)  \notin D(r)$, ensuring
(\ref{eq:topological-expansion}).

For any $\left(  \theta;x,y\right)  \in D\left(  r\right)  $, by
(\ref{eq:contraction-condition}), $\left\Vert \left(  1-\beta\right)
df^{k}\left(  \theta\right)  y\right\Vert <C\lambda^{k}\left\Vert y\right\Vert
<r$, so $h\left(  \beta,\left(  \theta;x,y\right)  \right)  \notin
D^{+}\left(  r\right)  $, which means that we have verified
(\ref{eq:topological-contraction}).

From (\ref{eq:homotopy-for-nhim}) we see that
 the map $\eta$ from Definition \ref{def:covering} is $\eta=\left(
f|_{\Lambda}\right)  ^{k}$. Since $\Lambda$ is invarant under $f,$ and $f$ is
a diffeomorphism, $\left(  f|_{\Lambda}\right)  ^{k}$ is also a
diffeomorphism, so $\deg_{2}\left(  \eta\right)  =1$ ensuring
(\ref{eq:deg-covering}).
Also from (\ref{eq:homotopy-for-nhim}), $A_{\theta}=df^{k}(\theta
)|_{E_{\theta}^{u}}$. Since $C\lambda^{k}<1$, by
(\ref{eq:expansion-condition-2}) $A_{\theta}$ is expanding. We have thus established (\ref{eq:l-covering}).

For sufficiently small $r$ the linearized dynamics inside $D\left(  r\right)
$ is topologically conjugate to the true dynamics in a neighborhood of
$\Lambda$, i.e.
$
l^{k}\circ g=g\circ f^{k}$
where $g$ is the conjugating homeomorphism \cite{Pugh1970}. The set
$D=g^{-1}\left(  D(r)\right)  $, equipped with the structure of the vector
bundle $E$ induced by $g$, constitutes the neighbourhood of $\Lambda$ in $M$
which $f^{k}$-covers itself.
\end{proof}

\section{Acknowledgements\label{sec:ack}}
We would like to thank Rafael de la Llave for his encouragement and suggestions. In particular, we thank him for pointing us to the Leray-Schauder continuation techniques and for his suggestions that led us to formulating Theorems \ref{th:forward-continuation} and \ref{th:continuation}.  We would also like to thank the anonymous Reviewers for their comments, suggestions and corrections, which helped us improve our paper.

\appendix
\section{Construction of the intersection number\label{sec:app}}

Here we present a brief overview of the construction of $I_{2}\left(
f,X,Z\right)  $.

\begin{lemma}
\label{lem:inter-finite} If $f$ is admissible, $f|_{X}$ is smooth and
transversal to $Z$, then the number $\#f^{-1}\left(  Z\right)  $ is finite.
\end{lemma}

\begin{proof}
Since $f$ is admissible, $f^{-1}(Z)\ $ is separated from $\partial X$ and
$f^{-1}\left(  \partial Z\right)  =\emptyset$. From the transversality of
$f|_{X}$ to $Z$ we obtain that $f^{-1}(Z)$ is a $0$-dimensional submanifold of
$X$. From transversality, the points in $f^{-1}(Z)$ cannot accumulate, and
since they are contained in the compact set $\overline{X}$, their number is finite.
\end{proof}

\begin{figure}[ptb]
\begin{center}
\includegraphics[height=3.3cm]{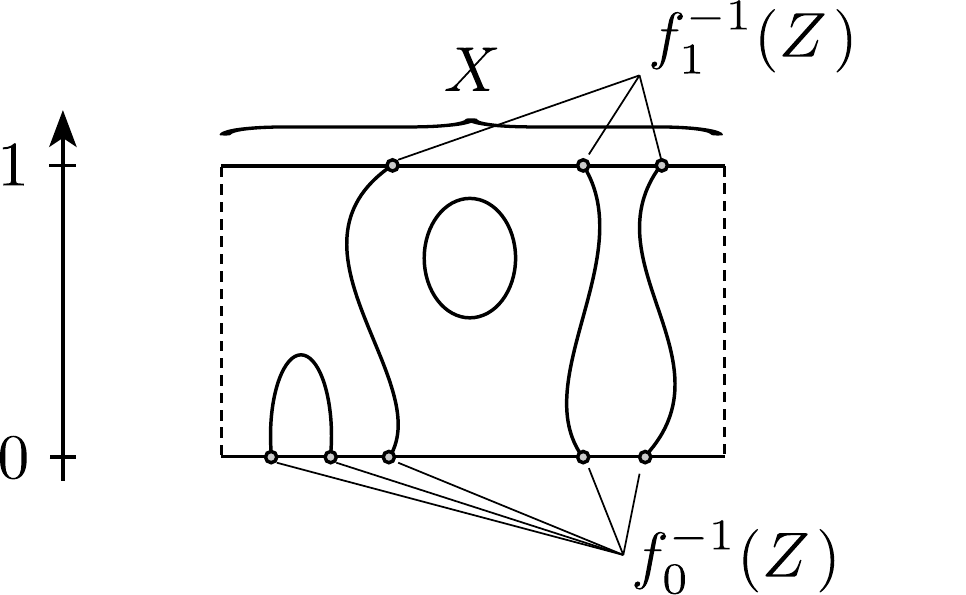}
\end{center}
\caption{Intuition behind the proof of Lemma \ref{lem:inter-same}. The
rectangle represents $[0,1]\times X$ and the curves contained in it are
$H^{-1}(Z)$. An important feature is that due to the admissibility of the
homotopy these curves cannot pass through $[0,1]\times\partial X$, which is
represented by the dotted lines. }%
\label{fig:inter-construction1}%
\end{figure}

Above we have shown that for $f|_{X}$ transversal to $Z$ defining the
intersection number as%
\begin{equation}
I_{2}\left(  f,X,Z\right)  :=\#f^{-1}\left(  Z\right)  \,\mathrm{mod}\,2
\label{eq:I2-for-transversal}%
\end{equation}
makes sense since we do not have infinity on the right hand side of the
defining equation. We now show that this number remains constant when passing
through an admissible homotopy. The proof of the following lemma is based on
the fact that if we have a homotopy $H$ between two smooth maps, both being
transversal to some given manifold, then we can make $H$ transversal to that
manifold by an arbitrarily small modification (see \cite[Extension Theorem and
Corollary that follows]{GP74} for details).

\begin{lemma}
\label{lem:inter-same}Assume that $f_{0},f_{1}:\overline{X}\rightarrow Y$ are
two admissible maps and that $f_{0}|_{X}$ and $f_{1}|_{X}$ are smooth and
transversal to $Z$. If $f_{0}$ and $f_{1}$ are homotopic via an admissible
homotopy, then
\[
\#f_{0}^{-1}(Z)=\#f_{1}^{-1}(Z)\,\mathrm{mod\,}2.
\]

\end{lemma}

\begin{proof}
The intuition behind the proof is given in Figure
\ref{fig:inter-construction1}.

Let $F:[0,1]\times\overline{X}\rightarrow Y$ be an admissible homotopy from
$f_{0}$ to $f_{1}$. By performing an arbitrarily small modification we can
arrive at an admissible $F$ such that $F|_{[0,1]\times X}$ is smooth and
transversal to $Z$. Note that since $F$ is admissible, $F^{-1}\left(
Z\right)  $ does not intersect $\left[  0,1\right]  \times\partial X$. Since
$F|_{[0,1]\times X}$ is transversal to $Z$, we have that $F^{-1}\left(
Z\right)  $ is a $1$-dimensional submanifold with boundary of $[0,1]\times X$,
the boundary being
\[
\partial F^{-1}(Z)=\{0\}\times f_{0}^{-1}(Z)\,\cup\,\{1\}\times f_{1}%
^{-1}(Z).
\]
By the classification of $1$-manifolds \cite{GP74}, $\partial F^{-1}(Z)$
consists of an even number of points, hence
\[
\#f_{0}^{-1}(Z)=\#f_{1}^{-1}(Z)\,\mathrm{mod}\, 2,
\]
as required.
\end{proof}

For any admissible map $f$ we can find an admissible map $g,$ arbitrarily
close to $f$, for which $g|_{X}$ is smooth, and such that $f$ and $g$ are
homotopic through an admissible homotopy. If $g|_{X}$ is not transversal to
$Z$, then we can again perform an arbitrarily small modification to obtain
transversality of $g|_{X}$ to $Z$ . We can therefore define%
\begin{equation}
I_{2}\left(  f,X,Z\right)  :=\#g^{-1}(Z)\,\mathrm{mod\,}2, \label{eq:I2-def}%
\end{equation}
where $f$ and $g$ are as above. By Lemma \ref{lem:inter-finite} the number
$\#g^{-1}(Z)$ is finite an by Lemma \ref{lem:inter-same} the number
$\#g^{-1}(Z)\,\mathrm{mod\,}2$ does not depend on the choice of $g$, so
$I_{2}\left(  f,X,Z\right)  $ from (\ref{eq:I2-def}) is well-defined. 

What is left is to prove that for $I_{2}\left(  f,X,Z\right)  $ defined in
(\ref{eq:I2-def}) we have the \correction{comment 16}{homotopy property,} intersection property and the excision property.

\correction{}{
\begin{lemma} \label{lem:hom-prop}If $f_1,f_2$ are homotopic through an admissible homotopy then $I_2(f_1,X,Z)=I_2(f_2,X,Z)$.
\end{lemma}
\begin{proof}
Since $f_1,f_2$ are homotopic through an admissible homotopy, they are admissible. As in the construction leading to (\ref{eq:I2-def}) we can find two smooth admissible maps $g_1$ and $g_2$, homotopic through an admissible homotopy to $f_1$ and $f_2$, respectively, for which $g_1|_{X}$ and $g_2|_{X}$ are transversal to $Z$. Since $g_1$ and $g_2$ are homotopic through an admissible homotopy (which is a composition of admissible homotopies: $g_1$ to $f_1$, $f_1$ to $f_2$, and $f_2$ to $g_2$) from (\ref{eq:I2-def}) and Lemma \ref{lem:inter-same} we obtain
\[I_2(f_1,X,Z)=\#g_1^{-1}(Z) \mbox{ mod } 2=\#g_2^{-1}(Z) \mbox{ mod } 2=I_2(f_2,X,Z),\]
as required.
\end{proof}}

\begin{lemma}
Let $f$ be an admissible map. If $I_{2}\left(  f,X,Z\right)  \neq0$ then
$f\left(  X\right)  \cap Z$ is nonempty.
\end{lemma}

\begin{proof}
We will show that if $f\left(  X\right)  \cap Z=\emptyset$ then $I_{2}\left(
f,X,Z\right)  =0$.

Assume that $f\left(  X\right)  \cap Z=\emptyset$. By admissibility
$f(\partial X)\cap\overline{Z}=\emptyset$ and $f(\overline{X})\cap\partial
Z=\emptyset$, so then $f(\overline{X})\cap\overline{Z}=\emptyset.$

We can approximate $f|_{X}$ by an arbitrarily close smooth map $g: X
\rightarrow Y$, homotopic to $f|_{X}$. We can extend this $g$ to $\overline
{X}$ in the natural way to obtain $g:\overline{X}\rightarrow Y$. We can take
this $g$ close enough so that it is homotopic to $f$ by an admissible homotopy
and $g(\overline{X})\cap\overline{Z}=\emptyset.$ Since the intersection of
$g(X)$ with $Z$ is empty and $g|_{X}$ is smooth, it is transversal to $Z$ (an
empty intersection is by definition transversal), and%
\[
I_{2}\left(  f,X,Z\right)  =I_{2}\left(  g,X,Z\right)  =\#g^{-1}\left(
Z\right)  \,\mathrm{mod}\,2=0,
\]
as required.
\end{proof}

\begin{lemma}
Let $f:X\rightarrow Y$ be an admissible map. If $V$ is an open subset of $X$
such that $f\left(  X\right)  \cap Z=f\left(  V\right)  \cap Z$, and $f\left(
\partial V\right)  \cap\overline{Z}=\emptyset$ then
\[
I_{2}(f,X,Z)=I_{2}\left(  f|_{\overline{V}},V,Z\right)  .
\]

\end{lemma}

\begin{proof}
We see that $f|_{\overline{V}}$ is admissible since%
\[
f\left(  \overline{V}\right)  \cap\partial Z\subset f\left(  \overline
{X}\right)  \cap\partial Z=\emptyset.
\]

We can find a $g$ arbitrarily close to $f$, homotopic through an admissible
homotopy (admissible both for $X$ and $V$) so that $g|_{X}$ is smooth. If
$g|_{X}$ is not transversal to $Z$, then we can make an arbitrarily small
modification of $g|_{X}$ to make it transversal. We can take $g$ close enough
to $f$ so that $g\left(  X\right)  \cap Z=g\left(  V\right)  \cap Z$. Since
$g|_{X}$ is transversal to $Z$, $g|_{V}$ is also transversal to $Z$. From
(\ref{eq:I2-for-transversal}--\ref{eq:I2-def})
\begin{eqnarray*}
I_{2}(f,X,Z)&=I_{2}(g,X,Z)=\#g^{-1}\left(  Z\right)  \,\mathrm{mod}%
\,2 \\ & \qquad \qquad=I_{2}\left(  g|_{\overline{V}},V,Z\right)  =I_{2}\left(  f|_{\overline{V}%
},V,Z\right)  ,
\end{eqnarray*}
as required.
\end{proof}

\correction{}{\section{Code for the computer assisted proof\label{sec:app-code}}
The program validates that $D\overset{f}{\Longrightarrow}D$. We write out the code and follow with comments.}

\begin{Verbatim}[fontsize=\footnotesize,commandchars=\!\%\@]
#include <iostream> 
#include "capd/capdlib.h" !h%1@
using namespace std; using namespace capd;

const interval mu=interval(1)/interval(10);

interval part(interval x,int N, int k) !h%2@
   { return x.left()+k*(x.right()-x.left())/N+(x-x.left())/N; }
    
interval hx(interval alpha,interval x,interval y) !h%3@
   { return alpha*2*x+(1-alpha)*(-8*x/5+4*power(x,3)+x*y/2); }
   
interval hy(interval alpha,interval theta,interval x,interval y) !h%3@
   { return (1-alpha)*(mu*y+2*sin(theta)/5+x*cos(theta)); }
   
bool ExitCondition(interval alpha,interval Bu,interval Bs) !h%4@
{   
   if(not(hx(alpha,Bu.left() ,Bs)<Bu)) return 0;
   if(not(hx(alpha,Bu.right(),Bs)>Bu)) return 0;
   return 1;
}
bool EntryCondition(interval alpha,interval theta,interval Bu,interval Bs,int N) !h%5@
{
   for(int i=0;i<N;i++)
   {
      for(int j=0;j<N;j++)
      {
         interval x=hx(alpha,part(Bu,N,i),part(Bs,N,j));
         interval y=hy(alpha,theta,part(Bu,N,i),part(Bs,N,j));
         if(not(x<Bu))
            if(not(x>Bu))
               if(not(subsetInterior(y,Bs))) return 0;
      }   
   }
   return 1;
}
int main()
{
   interval alpha=interval(0.0,1.0);
   interval Lambda=interval(2)*interval::pi()*interval(0.0,1.0);  
   interval Bu=interval(-1.0,1.0);     
   interval Bs=interval(-1.2,1.2); 
   for(int k=0;k<4;k++)
   {
      for(int i=0;i<100;i++)
      {
         if(ExitCondition(part(alpha,4,k),Bu,Bs)==0) return 0; !h%6@
         if(EntryCondition(part(alpha,4,k),part(Lambda,100,i),Bu,Bs,50)==0) return 0;
      }
   }
   cout << "proof complete" << endl; return 1; !h%7@
}
\end{Verbatim}

\correction{}{
\begin{itemize}
\item[\ch{1}] The code is based on the CAPD\footnote{Computer Assisted Proofs in Dynamics} library for {\tt C++}. To download and install the library follow the instructions found at http://capd.ii.uj.edu.pl.
\item[\ch{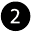}] This routine computes the {\tt k}-th part out of {\tt N} of the interval {\tt x}. The indexing is {\tt k=0,...,N-1}. For example, if  {\tt x=$[$1.0,2.0$]$} then for {\tt N=4} the {\tt 0}-th part is {\tt $[$1.0,1.25$]$} and the {\tt 3}-rd part is  {\tt $[$1.75,2.0$]$}.
\item[\ch{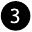}] These routines are used for the homotopy (\ref{eq:homotopy-CAP}) along the $x,y$ coordinates. Condition (\ref{eq:deg-covering}) follows directly from (\ref{eq:homotopy-CAP}) so we need to validate (\ref{eq:topological-expansion}--\ref{eq:topological-contraction}), for which local projections onto $x,y$ are sufficient. 
\item[\ch{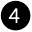}] 
We check that $h({\tt alpha}\times D^-)\cap D = \emptyset$ and return {\tt 1} if this is validated and {\tt 0} otherwise. This function will later be used to check (\ref{eq:topological-expansion}).
\item[\ch{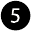}] This function is used to validate that $h({\tt alpha} \times D_{\tt theta})\cap D^+=\emptyset$. This is later used to validate (\ref{eq:topological-contraction}). The test is performed by subdividing $D_{\tt theta}$ into $N^2$ cubes and checking that the image of each of them does not intersect $D^+$.
\item[\ch{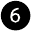}] This is the core of the program, where we validate (\ref{eq:topological-expansion}--\ref{eq:topological-contraction}). We do so by subdividing the parameter interval $[0,1]$ into four fragments and subdividing $\Lambda$ into $100$ parts. 
\item[\ch{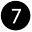}] Once the program reaches this point we are sure that all the needed conditions are validated. The program takes a fraction of a second, running on a standard laptop.
\end{itemize}}

\bibliography{bibl}

\end{document}